\documentclass{jhrs}
\usepackage{amsmath, amssymb, bm}
\usepackage{enumerate}
\usepackage[all]{xy}

\begin{document}
\title
{Cubical cospans and higher cobordisms \\
(Cospans in Algebraic Topology, III)}

\author {Marco Grandis}
\email{grandis@dima.unige.it}

\address{Dipartimento di Matematica
\\ Universit\`{a} di Genova
\\Via Dodecaneso 35
\\16146-Genova, Italy }

\thanks{Work supported by a research grant of Universit\`a di Genova.}



\received{}
\revised{}
\accepted{}

\newtheorem{thm}{Theorem}
\newtheorem{prop}{Proposition}
\newtheorem{definition}{Definition}
\newtheorem{defthm}{Definition and Theorem}
\newtheorem{defprop}{Definition and Proposition}
\newtheorem{lemma}{Lemma}
\newtheorem{cor}{Corollary}

% typographical items (colon, spaces, scriptstyle, url)
\def \c {\colon}
\def \cc {\, \colon \!}
\def \q {\qquad}
\def \qq {\qquad \qquad}
\def \qqq {\qquad \qquad \qquad \qquad}
\def \bu {{\scriptscriptstyle\bullet}}
\def \Bu {~ {\scriptscriptstyle\bullet} ~}
\def \ndt {\noindent}
\def \Ndt {\smallskip  \noindent}
\def \sst  {\scriptstyle}

\def\@url#1{\texttt{#1}}
{\catcode`\/=\active\gdef/{\char47\penalty-100 }\catcode`\.=\active
\gdef.{\char46\penalty-100 }}

\def\url{\catcode`\/=\active \catcode`\.=\active\catcode`\~=12 \@url}

% arrows
\def \lw {\, \leftarrow \,}
\def \dw {\downarrow}
\def \rrw {\, \rightrightarrows \,}
\def \rlw {\, \rightleftarrows \,}

\def \lrl {\;\; \raisebox{0.9ex}{$\longleftarrow$} \hspace{-3.7ex} 
 \raisebox{0.2ex}{$\longrightarrow$}
 \hspace{-3.7ex} \raisebox{-0.5ex}{$\longleftarrow$} \;\;}
 
\def \rlr { \;\; \raisebox{0.9ex}{$\longrightarrow$} \hspace{-3.7ex} 
 \raisebox{0.2ex}{$\longleftarrow$}
 \hspace{-3.7ex} \raisebox{-0.5ex}{$\longrightarrow$} \;\;}
 
\def \uw {\; \uparrow \!\!}
\def \todot {\; \cdot \hspace{-2.4ex} \to}

% operations, relations, logics, universe U, parts (spaces modified)
\def \ti {\! \times \!}
\def \te {\otimes}
\def \sdp {\rtimes}
\def \ci {\!\mathbin{\raise.3ex\hbox{$\scriptscriptstyle\circ$}}\!}
\def\le{\leqslant}
\def\ge{\geqslant}
\def \Sum{\raise.35ex\hbox{$\scriptstyle{\sum}$} \, }
\def \iso {\: \cong \:}
\def \eq {\! \sim \!}
\def \U  {{\cal U}}
\def \P  {{\cal P}}

% operators (roman)
\def \id {{\rm id}}
\def \Im {{\rm Im}}
\def \op{^{{\rm op}}}

% greek letters  (boldface: \bm for greek!)
\def \al {\alpha}
\def \be {\beta}
\def \ga {\gamma}
\def \de {\delta}
\def \ep {\epsilon}
\def \th {\vartheta}
\def \io {\iota}
\def \ka {\kappa}
\def \la {\lambda}
\def \si {\sigma}
\def \ph {\varphi}
\def \om {\omega}
\def \bal {{\bm \alpha}}
\def \bsi {{\bm \sigma}}

% categories (boldface, possibly with roman prefix)
\def \A {{\bf A}}
\def \B {{\bf B}}
\def \X {{\bf X}}
\def \Y {{\bf Y}}
\def \Set {{\bf Set}}
\def \Top {{\bf Top}}
\def \Cat {{\bf Cat}}
\def \ptCat {{\rm pt}{\bf Cat}}
\def \cCat {{\rm c}{\bf Cat}}
\def \Ab {{\bf Ab}}
\def \Cub {{\rm Cub}}
\def \pCub {{\rm pCub}}
\def \CUB {{\bf Cub}}
\def \nCUB {n{\bf Cub}}
\def \Cosp {{\rm Cosp}}

% small categories, multi-indices, other boldface characters
\def \bo {{\bf 1}}
\def \bt {{\bf 2}}
\def \s {{\bf s}}
\def \t {{\bf t}}
\def \We {{\bm \wedge}}
\def \Ve {{\bm \vee}}
\def \E {{\bf E}}

% double or cubical categories (blackboard boldface initial)
\def \AA {\mathbb{A}}
\def \BB {\mathbb{B}}
\def \II {\mathbb{I}}
\def \MM {\mathbb{M}}
\def \KK {\mathbb{K}}
\def \EE {\mathbb{E}}
\def \CCosp {\mathbb{C}{\rm osp}}
\def \CCblc {\mathbb{C}{\rm blc}}
\def \CCc {\mathbb{C}{\rm c}}
\def \pCCc {{\rm p}\mathbb{C}{\rm c}}
\def \CCOSP {\mathbb{C}{\rm OSP}}
\def \CCOB {\mathbb{C}{\rm OB}}
\def \CCob {\mathbb{C}{\rm ob}}
\def \CCub {\mathbb{C}{\rm ub}}
\def \RRel {\mathbb{R}{\rm el}}
\def \SSp {\mathbb{S}{\rm p}}

% standard (structured) sets (boldface)
\def \es  {\emptyset}
\def \sing  {\{*\}}
\def \N {{\bf N}}
\def \Z {{\bf Z}}
\def \I {{\bf I}}
\def \R {{\bf R}}

% abbreviations for cubical sets
\def \dd {\partial}
\def \ddp {\partial^+}
\def \ddm {\partial^-}
\def \dda {\partial^\alpha}
\def \ddb {\partial^\beta}

% abbreviations for cospans:  ^-, ^0, ^+
\def \um {u^-}
\def \up {u^+}
\def \Um {U^-}
\def \Up {U^+}
\def \Xm {X^-}
\def \Xz {X^0}
\def \Xp {X^+}
\def \vm {v^-}
\def \vp {v^+}
\def \Vm {V^-}
\def \Vp {V^+}
\def \Ym {Y^-}
\def \Yz {Y^0}
\def \Yp {Y^+}
\def \fm {f^-}
\def \fz {f^0}
\def \fp {f^+}
\def \dm {d^-}
\def \dpp {d^+}
\def \jm {j^-}
\def \jp {j^+}
\def \hm {h^-}
\def \hp {h^+}

\begin{abstract}
After two papers on weak cubical categories and {\it collarable} cospans, respectively, we put 
things together and construct a {\it weak} cubical category of cubical {\it collared} cospans of 
topological spaces. We also build a second structure, called a {\it quasi} cubical category, formed of 
arbitrary cubical cospans concatenated by homotopy pushouts. This structure, simpler but weaker, has 
{\it lax} identities. It contains a similar framework for cobordisms of manifolds with corners and could 
therefore be the basis to extend the study of TQFT's of Part II to higher cubical degree.

\end{abstract}

\maketitle

\section*{Introduction}\label{Intro}

	This is a sequel to two papers, cited as Part I \cite{G2} and Part II \cite{G3}. A reference I.2, or 
I.2.3, relates to Section 2 or Subsection 2.3 of Part I. Similarly for Part II.

	In Part I we constructed a cubical structure of higher cospans $ \CCosp_*(\X) $ on a category $ \X 
$ with pushouts, and abstracted from the construction the general notion of a weak cubical category.

	An $n$-cubical cospan in $\X $ is defined as a functor $ u \c \We^n \to \X$, where $ \We $ is the 
category
%
%formal cospan category
    \begin{equation}
    \We :	\qq	-1  \to  0  \lw  1  \qq   \text{(the  {\it formal cospan}).}
    \label{0.1}  \end{equation}

	These diagrams form a cubical set, equipped with compositions $ u +_i v $ of $i$-consecutive 
$n$-cubes, for $ i = 1,..., n. $ Such cubical compositions are computed by pushouts, and behave 
`categorically' in a weak sense, up to suitable comparisons.

	To make room for the latter, the $n$-th component of $ \CCosp_*(\X) $
%
%formal cospan category
    \begin{equation}
    \CCosp_n(\X)  =  \Cat(\We^n, \X),
    \label{0.2}  \end{equation}
is not just the {\it set} of functors 
$ u \c \We^n \to \X $ (the {\it $n$-cubes} of the structure), but the {\it 
category} of such functors and their natural transformations 
$ f \c u \to u' \c \We^n \to \X $ (the {\it $n$-maps} of the structure). 
The comparisons are {\it invertible} $n$-maps; but general $n$-maps are 
also important, e.g. to define limits and colimits (I.4.6, II.1.3). Thus, a {\it weak cubical category} has 
countably many {\it weak} (or {\it cubical}) directions $ i = 1, 2,..., n,... $ all of the same sort, and one 
{\it strict} (or {\it transversal}) direction, which is generally of a different sort.

	Truncating cubes and transversal maps at cubical degree 1, we get the {\it weak double category} $ 
\CCosp(\X) = 1\CCosp_*(\X), $ consisting of cospans and their natural transformations, with one strict 
composition (the transversal one) and one weak composition (by pushouts). Truncating at cubical 
degree 2, we get a structure related to Morton's construction for 2-cubical cospans (cf. \cite{Mo} and 
\ref{1.7}).

	Now, the weak cubical category $ \CCosp_*(\Top) $ of cubical cospans of topological spaces is 
not well suited for Algebraic Topology. Indeed, the composition by ordinary pushouts is not 
homotopically stable, and is not preserved by (co)homotopy or (co)homology functors, even in a weak 
sense.
	
	This is why, in Part II, {\it working in cubical degree 1}, we considered {\it collarable cospans}, 
forming a weak double category $ \CCblc(\Top) \subset \CCosp(\Top). $ Indeed, a push-out of 
collarable maps is a {\it homotopy pushout} (Thm. II.2.5). Therefore, cohomotopy functors induce 
`functors' from collarable cospans to spans of sets, providing - by linearisation - topological quantum 
field theories (TQFT) on manifolds and their cobordisms (II.3). Similarly, (co)homology and 
homotopy functors take collarable cospans to relations of abelian groups or (co)spans of groups, 
yielding other algebraic invariants (II.4).

	Notice that, as motivated in II.1.6, the definition of a collarable cospan is more general than one 
might expect. Indeed, we want to include the degenerate cospan $ e_1(X) = (\id X, \id X), $ instead of 
replacing it with the {\it cylindrical degeneracy} $ E_1(X), $ as usually done for cobordisms:
%
%  name
    \begin{equation}
E_1(X)  =  (\dm \c X \to X\ti [0, 1] \lw X  \cc \dpp),	\;\;\;\;
\dm(x)  =  (x, 0),  \;\;    \dpp(x)  =  (x, 1).
    \label{0.3}  \end{equation}

	The main problem is that such degeneracies do {\it not} satisfy an axiom of cubical sets: we only get 
$ E_1E_1(X) \iso E_2E_1(X) $ (see \ref{5.2}).

	Therefore, according to our definition in Part II, {\it a collarable cospan decomposes into a sum of 
a trivially collarable part} (a pair of homeomorphisms) {\it and a 1-collarable part}; only the second 
does admit collars (which are embeddings of cylinders, with disjoint images). However, $ e_1(X) $ and 
$ E_1(X) $ are {\it weakly equivalent}, as defined in II.2.8 (and here, in \ref{5.3}); {\it homotopy 
invariance} of functors on topological cospans is defined with respect to this notion.

	In the present paper, we combine the previous Parts. The first main construction, in Sections 2-4, is 
a weak cubical category $ \CCc_*(\Top) $ of {\it collared cospans}. It extends to unbounded cubical 
degree a weak double category $ \CCc(\Top), $ which is a variant of $ \CCblc(\Top) $ where collars 
are {\it assigned}. 
	
	The second main framework, in Section 5, is simpler, but satisfies weaker axioms on degeneracies 
and requires more comparisons. We work now with {\it arbitrary} cospans (not supposed to have 
collars) and replace:

- the ordinary degeneracies $ e_i $ with the cylindrical ones, $ E_i$,

- the ordinary concatenations with the {\it cylindrical} ones, constructed by means of homotopy 
pushouts.

	Notice that the latter is, in itself, a homotopy-invariant operation, which explains why {\it here} 
collars are not needed. As a price for this simplicity, we obtain a symmetric {\it quasi} cubical category 
$ \CCOSP_*(\Top), $ where degeneracies only satisfy the cubical relation mentioned above {\it up to 
isomorphism}, and behave as {\it lax} identities: the (left and right) unit comparisons are not invertible. 
An extensive study of motivations and goals of the homotopical weakening of identities can be found 
in J. Kock \cite{Ko}; see also Joyal-Kock \cite{JK} (and other papers in preparation, by the same authors).

	This new structure is made precise at the end, in Section 7, while in Section 6 we construct the {\it 
quasi} cubical category $ \CCOB_*(k) \subset \CCOSP_*(\Top) $ of $k$-manifolds and their 
cubical cobordisms, based on the notion of {\it manifolds with corners} \cite{Do, La, Ja}. Extending 
the cohomotopy functors to these structures (after II.3) should yield higher TQFTs, but this is not 
dealt with here. The 2-cubical truncation of our construction is related with the construction of Morton 
and Baez \cite{Mo, Ba} (which use cylindrical identities and pushout-concatenation). 

	As discussed in \ref{6.1}, we do not endeavour to construct a {\it weak} cubical category $ 
\CCob_*(k)$ $ \subset \CCc_*(\Top) $ based on the first main construction. It should be possible, but 
so heavy that the goal of getting a better behaviour of degeneracies might not justify its complication; 
moreover, considering the importance of `units-up-to-homotopy' in modelling homotopy types 
\cite{Ko, JK}, one may question the interest of this goal. (Notice also that in Part II we defined a weak 
double subcategory $ \CCob(k) \subset \CCblc(\Top) $ of {\it $k$-dimensional manifolds and their 
cobordisms}, based on the fact that `cobordisms are always collarable'.)

	Size problems can be dealt with as in Part I, using a hierarchy of two universes, $ \U_0 \in  \U$. 
Small category means $\U$-small. The constructions $ \CCosp_*(-)$, $ \CCc_*(-)$, etc. apply to the 
small category $ \Top $ of $\U_0$-small spaces. $ \Cat $ is the 2-category of $\U$-small categories, 
to which $ \Top $ belongs. Finally, the index $ \al  $ takes values $ \pm 1, $ written $ \pm $ in 
superscripts, and $ IX $ denotes the standard cylinder $ X \ti  [0, 1] $ on a space $ X$.

\section{The symmetric weak cubical category of cospans}\label{1}

	We begin by recalling cubical cospans, from Part I; we also introduce symmetric {\it quasi} cubical 
sets, which will be used later to define quasi cubical categories.

\subsection{Symmetric cubical and quasi cubical sets}\label{1.1} A cubical set 
$ ((A_n), (\dda_i)), (e_i)) $ has {\it faces} $ (\dda_i) $ and {\it degeneracies} $ (e_i)$
%
%  faces and degeneracies
    \begin{equation}
\dda_i \c A_n     \rlw        A_{n-1}  \cc e_i  	\qq	 (i = 1,..., n;  \; \al  = \pm),
    \label{1.1.1}  \end{equation}
satisfying the cubical relations:
%
%  cubical relations 1
    \begin{equation}
\dda_i.\ddb_j  =  \ddb_j.\dda_{i+1}  \qq    (j \le i),
    \label{1.1.2}  \end{equation}
%
%
%  cubical relations 2
    \begin{equation}
e_j.e_i  =  e_{i+1}.e_j   \qq	(j \le i),
    \label{1.1.3}  \end{equation}
%
%
%  cubical relations 3
    \begin{equation}
\dda_i.e_j  =  e_j.\dda_{i-1}  \;\;  (j < i), \q	{\it or}   \;\;   \id   \;\;   (j = i),  \q
{\it or}   \;\;   e_{j-1}.\dda_i   \;\;  (j > i).
    \label{1.1.4}  \end{equation}

	As in I.2.2, a {\it symmetric} cubical set is a cubical set which is further equipped with {\it 
transpositions}
%
%  transpositions
    \begin{equation}
	s_i \c A_n \to A_n	\qq	(i = 1,..., n-1),
    \label{1.1.5}  \end{equation}
%(
satisfying the Moore relations (see Coxeter-Moser \cite{CM}, 6.2; or Johnson \cite{Jo}, Section 5, 
Thm. 3)
%
%  Moore relations
    \begin{equation}
s_i.s_i  =  1,   \q
s_i.s_ j.s_i  =  s_ j.s_i.s_ j  \;\;  (i = j-1),
 \q   s_i.s_ j  =  s_ j.s_i  \;\;   (i < j-1),
    \label{1.1.6}  \end{equation}
 and the following equations:
%
% consistency of transpositions
    \begin{equation}
    \begin{array}{cccccccccc}
\qq   &\;\;&      j < i  &\q&      j = i    &\q&       j = i+1  &\q&      j > i+1
\\[5pt]
\dda_j.s_i   \;\; =   &&   s_{i-1}.\dda_j   &&   \dda_{i+1}   &&   \dda_i   &&   s_i.\dda_j
\\[4pt]
s_ i.e_ j	\;\;  =  	 &&   e_ j.s_ {i-1}   &&   e_ {i+1}   &&   e_ i   &&   e_ j.s_ i.
    \label{1.1.7}\end{array}
    \end{equation}

		We will speak of a {\it symmetric quasi cubical set} when the axiom (\ref{1.1.3}), on pure degeneracies, is 
omitted, which will be important for `cylindrical degeneracies' (cf. \ref{5.2}).

	Actually, the presence of transpositions makes all faces and degeneracies determined by the ones 
belonging to a fixed direction, e.g. the 1-directed ones, $ \dda_1 $ and $ e_1. $ In fact, from $ 
\dda_{i+1} = \dda_i.s_i $ and $ e_{i+1} = s_i.e_i, $ we deduce that:
%
%  faces and degeneracies determined by the ones in direction 1
    \begin{equation}
\dda_i  =  \dda_1.\s_i ',	\q   e_i  =  \s_i.e_1	\qq    (i = 2,..., n;  \;  \al  = \pm ),
    \label{1.1.8}  \end{equation}
where we are using the inverse `permutations':
%
%  faces and degeneracies determined by the ones in direction 1
    \begin{equation}
\s_i  =  s_{i-1}. ... .s_1,	\qq	\s_i '  =  s_1. ... .s_{i-1}.
    \label{1.1.9}  \end{equation}

\subsection{A reduced presentation}\label{1.2} 
These relations lead to a more economical presentation of our structures.

\begin{prop}
A symmetric {\it quasi} cubical set can be equivalently defined as a system
%
%  reduced presentation of a symmetric {\it quasi} cubical set
    \begin{equation}
    \begin{array}{l}
A  =  ((A_n), ( \dda_1), (e_1), (s_i)),
\\[4pt]
 \dda_1  \c A_n   \rlw  A_{n-1}  \cc e_1,	\q  
 s_1 \c A_{n+1} \to A_{n+1}   \qq	(n \ge 1),
    \label{1.2.1}\end{array}
    \end{equation}
under the Moore relations (\ref{1.1.6}) and the axioms:
%
%  reduced axioms in the quasi case
    \begin{equation}
    \begin{array}{ll}
\dda_1.\ddb_1  =  \ddb_1.\dd_1^\al.s_1,  \qq   &	\dda_1.e_1  =  \id ,
\\[3pt]
s_i.\dda_1  =  \dda_1.s_{i+1},  &    e_1.s_i  =  s_{i+1}.e_1.
    \label{1.2.2}\end{array}
    \end{equation}
	For a symmetric cubical set one adds the axiom:
%
%  additional axiom
    \begin{equation}
e_1e_1 = s_1.e_1e_1   \qq 	\text{({\it symmetry of second-order degeneracies}).}
    \label{1.2.3}  \end{equation}
\end{prop}
\begin{proof}
Defining the other faces and degeneracies by (\ref{1.1.8}), one can prove the global cubical relations.

	For instance, letting $ j \le i$, we have:

$$
\dda_i. \ddb_j   =   \dda_1 .(s_1. ... .s_{i-1}). \ddb_1 .(s_1. ... .s_{j-1})  
=   \dda_1  \ddb_1 .(s_2. ... .s_i).\s_j '   =   \ddb_1  \dda_1 .(\s_{i+1} '  \s_j ' ),  
$$
$$
\ddb_j .\dda_{i+1}  =   \ddb_1 .(s_1. ... .s_{j-1}). \dda_1 .\s_{i+1} '  =  \ddb_1 \dda_1 .(s_2. ... .s_j).\s_{i+1} '  =  \ddb_1 \dda_1 .(s_1 \s_{j+1} '  \s_{i+1} ').
$$
	Now, it suffices to verify that the two operators at the end of these equalities coincide. In fact, in the 
symmetric group $ S_n $ of permutations of the set $ \{1,..., n\}$, $ \s_i ' $ is identified with the 
permutation:
$$
\bsi _i  =  (i, 1,..., \hat{i},..., n).
$$

	Thus, always for $ j \le i, $ the operators $ \s_{i+1}\s_j $ and $ \s_{j+1}\s_{i+1} $ correspond to the 
permutations:
$$
\bsi _{i+1}\bsi _j  =  (i+1, j, 1,..., \hat{j},..., (i+1)\hat{},..., n),
$$  $$
\bsi _{j+1}\bsi _{i+1}  =  (j, i+1, 1,..., \hat{j},..., (i+1)\hat{},..., n),
$$	
and the transposition $ s_1 $ turns one permutation into the other.\end{proof}

\subsection{A setting for cospans}\label{1.3}
   The model of our construction of cubical cospans, in Part I, is the {\it 
formal cospan} category $ \We $ and its Cartesian powers $ \We^n $ $(n \ge 0)$
%
%  diagram for \We and \We^2
    \begin{equation}
    \begin{array}{c}  \xymatrix  @C=10pt @R=20pt
{
&&&&
~(-1,-1)~   \ar[r]    \ar[d]  &  ~(0,-1)~   \ar[d]  &  ~(1,-1)~   \ar[l]  \ar[d]  &
\Bu    \ar[r]^1   \ar[d]^2  &
\\ 
~-1~   \ar[r]     &  ~0~  &  ~1~   \ar[l]   &&
~(-1,0)~   \ar[r]     &  ~(0,0)~  &  ~(1,0)~   \ar[l]  &
\\  
&  \We  &&&
~(-1,1)~   \ar[r]    \ar[u]  &  ~(0,1)~   \ar[u]  &  ~(1,1)~   \ar[l]  \ar[u]
&  \We^2
}
    \label{1.3.1} \end{array} \end{equation}
\smallskip

	(Identities and composed arrows are always understood in such diagrams of finite categories.) 
Thus, an {\it $n$-cospan} in the category $ \X $ is a functor $ u \c \We^n \to \X. $ But, in order to be 
able to compose them, in direction $ i = 1,..., n, $ we need (a full choice of) pushouts in $ \X$.

	More generally, according to the terminology of Part I, a {\it pt-category}, or {\it category with 
distinguished pushouts}, is a ($\U$-small) category where {\it some} spans $ (f, g) $ have {\it one 
distinguished} pushout $ (f', g') $ (in a symmetric way, of course)
%
%  diagram for distinguished pushouts
    \begin{equation}
    \begin{array}{c}  \xymatrix  @C=10pt @R=10pt
{
\Bu   \ar[rr]^f    \ar[dd]_g  &&  \Bu   \ar[dd]^{f'}   & \qq &
~x~   \ar[rr]^f    \ar[dd]_1  &&  ~x'~   \ar[dd]^1
\\ 
&    \ar@{--}[r]   \ar@{--}[d] &   &&
&    \ar@{--}[r]   \ar@{--}[d] &
\\
\Bu   \ar[rr]_g    &&   \Bu  &&
~x~   \ar[rr]_f      &&  ~x'~  
}
    \label{1.3.2} \end{array} \end{equation}

\Ndt  and we assume the following {\it unitarity constraints}:

\Ndt   (i)	each square of identities is a distinguished pushout,

\Ndt   (ii) {\it if} the pair $ (f, 1) $ has a distinguished pushout, this is $ (1, f) $ (as in the right diagram 
above).

	A {\it pt-functor} $ F \c \X \to \Y $ is a functor between pt-categories which {\it strictly} preserves 
the distinguished pushouts. We speak of a {\it full} (resp. {\it trivial}) {\it choice}, or of a {\it 
category} $ \X $ {\it with full} (resp. {\it trivial}) {\it distinguished pushouts}, when all the spans in $ 
\X $ (resp. only the spans of identities) have a distinguished pushout. 

	The category $ \ptCat $ of pt-categories and pt-functors is $\U$-complete and $\U$-cocom-plete. 
For instance, the product of a family $ (\X_j)_{j\in J} $ of pt-categories indexed on a $\U$-small set is 
the Cartesian product $ \X $ (in $ \Cat), $ equipped with those (pushout) squares in $ \X $ whose 
projection in each factor $ \X_j $ is a distinguished pushout. In particular, the terminal object of $ 
\ptCat $ is the terminal category $ \bo $ with the unique possible choice: its only square is a 
distinguished pushout.

	$\Cat $ embeds in $ \ptCat, $ equipping a small category with the trivial choice of pullbacks (a 
procedure which is left adjoint to the forgetful functor $ \ptCat \to \Cat). $ Limits and colimits are 
preserved by this embedding. Our construction requires this sort of double setting $ \Cat \subset 
\ptCat, $ with `models' $ \We^n $ having a trivial choice and cubical cospans $ \We^n \to \X $ living in 
categories with a full choice (which is necessary to compose them). 

	Notice that the category $ \We^n $ {\it has} all pushouts; however, should we use the full choice 
suggested by diagram (\ref{1.3.1}), a pt-functor $ \We^2 \to \X $ would only reach very particular 2-cubical 
cospans. Notice also that, in the absence of the unitarity constraint (i) on the choice of pushouts, the 
terminal object of $ \ptCat $ would still be the same, but a functor $ \bo \to \X $ could only reach an 
object whose square of identities is distinguished. On the other hand, condition (ii) just simplifies 
things, making our units work strictly.

\subsection{The structure of the formal cospan}\label{1.4}  
The category $ \We $ has a basic structure of formal 
symmetric interval, with respect to the Cartesian product in $ \Cat $ (and $ \ptCat). $ This structure 
consists of two {\it faces} $ (\ddm, \ddp), $ a {\it degeneracy} $ (e) $ and a {\it transposition} 
symmetry $ (s)$
%
%  structure of the formal cospan
    \begin{equation}
    \begin{array}{ll}
\dda  \c \bo    \rrw    \We,	   \q   e \c \We \to \bo,	 \q &
 s \c \We^2 \to \We^2  \q (\al  = \pm 1),
\\[3pt]
\dda (*)  =  \al ,		&		s(t_1, t_2)  =  (t_2, t_1).
    \label{1.4.1}\end{array}
    \end{equation}

	(A functor with values in the ordered set $ \We^n $ is determined by its value on objects). 
Composition is - formally - more complicated. The {\it model of binary composition} is the pt-
category $ \We_2 $ displayed below, with one non-trivial distinguished pushout
	
%
%  model of binary composition
    \begin{equation}
    \begin{array}{c}  \xymatrix  @C=8pt @R=1pt  
{   
&&&&   ~0~    &&&&&&&  
\sing   \ar[rr]^{\ddp}   \ar[dddd]_{\ddm}       &&   ~\We~  \ar[dddd]^{k^-}
\\   
&&&   \ar@{--}[rd]   &&  \ar@{--}[ld]
\\
&&   ~a~ \ar[rruu]   &&&&   ~c~  \ar[lluu]
\\
\\
~ -1~   \ar[rruu]  &&&&   ~b~  \ar[lluu]  \ar[rruu]   &&&&   ~1~   \ar[lluu] &&& 
~\We~  \ar[rr]_{k^+}         &&   ~\We_2~  &
}
    \label{1.4.2} \end{array} \end{equation}

	Now, the commutative square at the right hand above is {\it not} a pushout; in fact, in $ \Cat $ or $ 
\ptCat, $ the corresponding pushout is the subcategory $ \We_{(2)} $ lying at the bottom of $ \We_2 $ 
in the diagram above:
    \begin{equation}
-1  \to   a   \lw   b   \to   c   \lw   1		\qq	\We_{(2)}.
    \label{1.4.3}  \end{equation}
	But the relevant fact is that a category $ \X $ with {\it full} distinguished pushouts `believes' that the 
square above is (also) a pushout. Explicitly, we have a - so to say - {\it para-universal} property of $ \We_2$:

\Ndt  (a) given two cospans $ u, v \c \We \to \X $ in a category with full distinguished pushouts, with
 $\ddp_1u = \ddm_1 v, $
  there is one pt-functor $ [u, v] \c \We_2 \to \X $ such that $ [u, v].k^- = u $ and $ [u, v].k^+ = v.$

	The {\it concatenation map} 
    \begin{equation}
	k \c \We \to \We_2,
    \label{1.4.4}  \end{equation}
is an embedding, already displayed above by the labelling of objects in $ \We_2.$

	As usual in formal homotopy theory, the functors 
$ \; (-)^n_i= \We^{i-1} \times  - \times  \We^{n-i} \c $ $ \ptCat  \to   \ptCat $ 
produce the higher structure of the interval, for $ 1 \le i \le n $ and $ \al  = \pm 1$
%
%  name
    \begin{equation}
    \begin{array}{ll}
\dda_i \c \We^{n-1} \to \We^n,	\qq &	   \dda_i(t_1,..., t_{n-1})  =  (t_1,..., \al ,..., t_{n-1}),
\\[3pt]
e_i \c \We^n \to \We^{n-1},	&	   e_i(t_1,..., t_n)  =  (t_1,..., \hat{t}_i,..., t_n),
\\[3pt]
s_i \c \We^{n+1} \to \We^{n+1},	   &    s_i(t_1,..., t_{n+1})  =  (t_1,..., t_{i+1}, t_i,..., t_n).
    \label{1.4.5}\end{array}
    \end{equation}

	Moreover, acting on (\ref{1.4.2}) and $ k, $ these functors yield the {\it n-dimensional i-concate-nation 
model} $ \We^{ni}_2 $ and the {\it n-dimensional i-concatenation map} $ k_i \c \We^n \to 
\We^{ni}_2 $ 

%
% i-concatenation
    \begin{equation}
    \begin{array}{lll}  \xymatrix  @C=20pt @R=15pt   
{ ~ \We^{n-1}~
\ar[r]^{\ddp_i} \ar[d]_{\ddm_i} & ~\We^n~   \ar[d]^{k^-_i}
&\q& 
\We^{ni}_2  =  \We^{i-1} \ti  \We_2 \ti  \We^{n-i},
\\
~\We^n~    \ar[r]_{k^+_i}  &  ~\We^{ni}_2~ &&
k_i = \We^{i-1} \ti  k \ti  \We^{n-i}\c \We^n \to \We^{ni}_2.
}
    \label{1.4.6} \end{array} \end{equation}

	Again, the square above is not a pushout, but $ \X $ (having full distinguished pushouts) believes it 
is.

\subsection{The symmetric pre-cubical category of cospans}\label{1.5}
A {\it symmetric pre-cubical category} 
    \begin{equation}
\A  =  ((A_n), (\dda_i), (e_i), (s_i), (+_i)),
    \label{1.5.1}  \end{equation}
is a symmetric cubical set with compositions, satisfying the consistency axioms (cub.1-2) of I.1.2, 
where the transpositions $ s_i $ agree with the compositions $ +_i $ (see I.2.3). Notice that we are not 
(yet) assuming that the cubical compositions $ +_i $ behave in a categorical way or satisfy 
interchange, in any sense, even weak.

	For a category $ \X $ with full distinguished pushouts, there is such a structure $ \A = \Cosp_*(\X). 
$ (Below, we will use a different notation for the {\it weak cubical category} $ \CCosp_*(\X), $ a 
richer structure whose components are categories instead of sets.) An $n$-cube, or {\it n-cubical 
cospan}, is a functor $ u \c \We^n \to \X; $ faces, degeneracies and transpositions are computed 
according to the formulas (\ref{1.4.5}) for the formal interval $ \We$  
$$
\Cosp_n(\X)  =  \Cat(\We^n, \X)  =  \ptCat(\We^n, \X),
$$
    \begin{equation}
    \begin{array}{ll}
\dda_i(u)  =  u.\dda_i \c \We^{n-1} \to \X,  &
 \dda_i(u)(t_1,..., t_{n-1})  =  u(t_1,..., t_{i-1}, \al ,..., t_{n-1}),	
\\[3pt]
e_i(u)  =  u.e_i \c \We^n  \to \X,   &
e_i(u)(t_1,..., t_n)  =  u(t_1,..., \hat{t}_i,..., t_n),
\\[3pt]
s_i(u)  =  u.s_i \c \We^{n+1} \to \X,   &
s_i(u)(t_1,..., t_{n+1})  =  u(t_1,..., t_{i+1}, t_i,..., t_{n+1}).
    \label{1.5.2}\end{array}
    \end{equation}

	The $i$-composition $ u +_i v $ is computed on the $i$-concatenation model $ \We^{ni}_2 $ 
(\ref{1.4.6}), as
%
%  i-composition
    \begin{equation}
	u +_i v  =  [u, v].k_i \c \We^n \to \We^{ni}_2 \to \X   \qq	(\ddp_i (u) = \ddm_i (v)).
    \label{1.5.3}  \end{equation}

	A {\it symmetric cubical functor} $ F \c \A \to \B $ between symmetric pre-cubical categories is a 
morphism of symmetric cubical sets which preserves all composition laws.

	For an ordinary (i.e., 1-cubical) cospan $ u \c \We \to \X, $ we write $ u = (\um, \up) \c \Xm \todot 
\Xp $ to specify its cubical faces (notice the dot-marked arrow).

\subsection{The weak cubical category of cospans}\label{1.6}
Now, starting from a category $ \X $ with full 
distinguished pushouts, we have a {\it symmetric weak cubical category} $ \CCosp_*(\X) $ (as 
defined in I.4), which is unitary (under the unitarity constraint \ref{1.3}(i)-(ii) in $ \X). $ It consists of 
the following data.

\Ndt (a) Our previous $ \Cosp_*(\X) $ forms the symmetric pre-cubical category of cubical objects.

\Ndt  (b) A {\it transversal n-map} $ f \c u \to u', $ is a natural transformations of $n$-cubes $ f \c u \to u' 
\c \We^n \to \X, $ or equivalently an $n$-cube in the pt-category $ \X^\bt $ of morphisms of $ \X $ 
(equipped with the coherent choice of distinguished pushouts). Transversal maps form a symmetric 
pre-cubical category $ \Cosp_*(\X^\bt), $ with:

    \begin{equation}
    \begin{array}{ll}
\dda_i(f)  =  f.\dda_i \c  u\dda_i \to u'\dda_i \c  \We^{n-1} \to \X	\qq   &   (i \le n, \;  \al  = \pm 1),
\\[3pt]
e_i(f)  =  f.e_i \c  u.e_i \to u'.e_i \c  \We^n \to \X	  &   (i \le n),
\\[3pt]
s_i(f)  =  f.s_i \c  u.s_i \to u'.s_i \c  \We^{n+1} \to \X	  &   (i \le n),
\\[3pt]
f +_i g  =  [f, g].k_i \c \We^n \to \X^\bt	  &   (\ddp_i f = \ddm_i g).
    \label{1.6.1}\end{array}
    \end{equation}

	Notice that an $n$-map should be viewed as {\it (n+1)}-dimensional, and will also be called an {\it 
(n+1)-cell}.
	
\Ndt  (c) The symmetric pre-cubical functors of {\it transversal faces} $(\dda_0) $ and {\it transversal 
degeneracy} $ (e_0) $ simply derive, contravariantly, from the obvious functors linking the categories 
$ \bo $ and $ \bt $ 
%
%  0-faces and 0-degeneracies
    \begin{equation}
e_0 \c  \Cosp_*(\X)     \lrl    \Cosp_*(\X^\bt)  \cc  \dd^\al_0
\qq   ( \dd^\al_0  \c \bo   \rlr         \bt \cc e_0,  \;\;  \al  = \pm).
    \label{1.6.2}    \end{equation}		

\Ndt   (d) The composition $ hf \c u \to u'' $ of transversally consecutive $n$-maps is the composition of 
natural transformations. It is categorical and preserves the symmetric cubical structure.

\Ndt  (e) The cubical composition laws behave categorically up to suitable comparisons for associativity and 
interchange, which are invertible {\it special} transversal maps. (A transversal $n$-map $ f $ is said to 
be {\it special} if its $ 2^n $ {\it vertices} $ f(\al _1,..., \al _n) $ are identities, for $ \al _i = \pm ; $ see I.4.1.)

	These comparisons are explicitly constructed in I.4.4, after a study of the structure of the models 
$ \We^n $ (see I.3.5, I.3.6).

\subsection{Truncation}\label{1.7}

Truncating everything at cubical degree $ n, $ we get the {\it symmetric weak 
(n+1)-cubical category} $ n\CCosp_*(\X), $ which contains the $k$-cubes and $k$-maps of $ 
\CCosp_*(\X) $ for $ k \le n $ (I.4.5). Indeed, its $n$-maps are `actually' (n+1)-dimensional cells.

	In the 1-truncated case $ 1\CCosp_*(\X) = \CCosp(\X) $ there is only one cubical direction and no 
transposition, and we drop the term `symmetric': a {\it weak 2-cubical category} amounts, precisely, to 
a weak (or pseudo) double category, as studied in \cite{GP1}-\cite{GP4}: a structure with a strict `horizontal' 
composition and a weak `vertical' composition, under strict interchange.

	The 2-truncated structure $ 2\CCosp_*(\X), $ a symmetric quasi 3-cubical category, is related to 
Morton's constructions \cite{Mo}: a `double bicategory' of 2-cubical cospans. Loosely speaking, and 
starting with $ 2\CCosp_*(\X), $ one should omit the transposition and restrict transversal maps to the 
special ones. See also \cite{Ba}.

\section{Collarable and collared cospans}\label{2}

	After recalling {\it collarable} (ordinary) cospans, from Part II, we define now {\it collared cospans,} 
a variation of the former (already hinted at in II.2.2), where collars are part of the data instead of just 
existing. We begin with the weak double category of {\it pre-collared} cospans $ \pCCc(\Top) $ and 
single out within the latter the weak double subcategory $ \CCc(\Top) $ of collared cospans. Higher 
cubical degree is deferred to Section 4.

\subsection{Topological embeddings. }\label{2.1}
Let us recall that a {\it topological embedding} is an injective map $ f 
\c X \to Y $ in $ \Top, $ where the space $ X $ has the pre-image topology. A closed injective map is 
necessarily an embedding, and will be called a {\it closed embedding}. Open embeddings are similarly 
defined, but not explicitly used here.

	It is easy to see that a pushout of topological embeddings (resp. closed embeddings) in $ \Top $ 
consists of topological embeddings (resp. closed embeddings) and is also a pullback.

	Decomposing a space $ X $ into a categorical sum $ X = \Sum X_j $ (in $ \Top) $ amounts to giving 
a partition of the space $ X $ into a family of {\it clopens} (closed and open subspaces).

\subsection{Collarable maps}\label{2.2}
	Now, we recall from Part II the construction of the weak double subcategory $ 
\CCblc(\Top) \subset \CCosp(\Top) $ of collarable cospans. Motivations for these definitions have 
been recalled in the Introduction.

	To begin with, a {\it collarable map} $ f \c X \to Y $ (II.2.1) is a continuous mapping which can be 
decomposed into a sum of two maps, so that: 
    \begin{equation}
f  =  f_0 + f_1 \c  X_0 + X_1 \to Y_0 + Y_1		\qq  \text{({\it collarable decomposition}),}
    \label{2.2.1}  \end{equation}

\Ndt   (i)  $f_0 \c X_0 \to Y_0 $ is a {\it homeomorphism}, also called a {\it trivially collarable}, or {\it 0-collarable} map,

\Ndt   (ii) $ f_1 \c X_1 \to Y_1 $ is a {\it 1-collarable} map, i.e. it admits a {\it collar} $ F $ which extends it 
    \begin{equation}
F \c  IX_1  \to Y_1,	\qq	f_1  =  F(-, 0) \c X_1 \to Y_1,
    \label{2.2.2}  \end{equation}
to the cylinder $ IX_1 = X_1 \ti  [0, 1] $ and is a {\it closed embedding} (\ref{2.1}), with $ F(X_1 \ti  
[0, 1[) $ {\it open in} $ Y_1.$

	Both conditions are only satisfied by $ \id\es. $ Omitting empty components, an {\it irreducible} 
collarable decomposition is uniquely determined. Collarable maps are also closed embeddings; they 
are {\it not} closed under composition in $ \Top $ (II.2.9).

\subsection{Collarable cospans}\label{2.3}
 Limits and colimits in the weak double category $ \CCosp(\Top) $ are 
recalled in II.1.3. As defined there, a {\it sub-cospan} of a topological cospan
    \begin{equation}
u  =  (\um \c \Xm \to \Xz \lw \Xp \cc \up),
    \label{2.3.1}  \end{equation}
is a regular subobject of $ u, $ i.e. an equaliser of two 1-maps $ u \to v $ in $ \CCosp(\Top). $ It 
amounts to assigning three subspaces $ (\Ym, \Yz, \Yp) $ such that
    \begin{equation}
Y^t \subset  X^t,   \qq   u^\al (Y^\al )  \subset  \Yz   \qq   (t \in  \We;  \; \al  = \pm ).
    \label{2.3.2}  \end{equation}
	We say that the sub-cospan is {\it open} (resp. {\it closed}) in $ u $ if so are the three subspaces $ 
Y^t\subset X^t . $ The sub-cospans of $ u $ form a complete lattice, which is a sublattice of $ ~ \P(\Xm) \ti  
\P(\Xz) \ti  \P(\Xp).$

	Thus, to give a decomposition $ u = \Sum u_j $ into a sum of cospans amounts to give a {\it clopen 
partition} $ (u_j) $ of $ u, $ i.e. a cover of $ u $ by disjoint sub-cospans, {\it closed and open} in $ u 
$ (II.1.3).

	A {\it collarable cospan} (II.2.2) of topological spaces is a topological cospan $ u \c \Xm \todot 
\Xp $ which {\it admits a collarable decomposition}, i.e. can be decomposed in a binary sum, where
%
%  collarable decomposition of a cospan
    \begin{equation}
    \begin{array}{ll}
u  =  u_0 + u_1 = (\Xm_0 + \Xm_1  \to \; \Xz_0 + \Xz_1  \lw  \Xp_0 + \Xp_1),
\\[3pt]
u_0  =  (\um_0 \c   \Xm_0 \to  \Xz_0  \lw  \Xp_0  \cc  \up_0), 	\;
u_1  =  (\um_1 \c   \Xm_1 \to  \Xz_1  \lw  \Xp_1  \cc  \up_1),
    \label{2.3.3}\end{array}
    \end{equation}

\Ndt   (i)  $ u_0 = (\um_0, \up_0) $  {\it is a pair of homeomorphisms}, also called a {\it trivially collarable}, or {\it 
0-collarable} cospan,
	
\Ndt   (ii) $ u_1 = (\um_1, \up_1) $ is a {\it 1-collarable cospan}; by this we mean that it {\it admits a collar 
cospan} (or, simply, a {\it collar}), i.e. a cospan $ (\Um, \Up) $ formed of a pair of collars of its maps 
having {\it disjoint images}. In other words, we have two disjoint closed embeddings
%
% collar cospan
    \begin{equation}
U  =  (\Um \c  I\Xm_1 \to \Xz_1 \lw I\Xp_1  \cc \Up),	\q	u^\al_1  =  U^\al (-, 0) \c X^\al_1 \to \Xz,
    \label{2.3.4}  \end{equation}
where $ U^\al (X^\al_1 \ti  [0, 1[) $ is open in $ \Xz_1$.

\smallskip
	Both conditions are only satisfied by the empty cospan $ e_1(\es). $ Omitting empty components, 
an {\it irreducible} collarable decomposition is uniquely determined. The maps $ u^\al  $ are also 
closed embeddings; in the 1-collarable case, they have disjoint images. 

	Cubical faces and degeneracy are inherited from  $\CCosp(\Top)$
    \begin{equation}
\dda_1 u  =  X^\al,    \q   e_1(X)  =  (\id  \c X\to X \lw X \cc \id)   \q   (\al  = \pm 1).
    \label{2.3.5}  \end{equation}
	We have proved (Thm. II.2.4) that collarable cospans are closed under concatenation: given a 
consecutive collarable cospan $ v \c \Ym \todot \Yp $ (with $ \Xp = A = \Ym), $ the cospan
 $ w = u +_1 v $
 decomposes in a trivial part, computed by a pushout of homeomorphisms (over the clopen subspace 
$ \Xp_0 \cap \Ym_0 $ of $ A $ on which $ u $ and $ v $ are {\it both} 0-collarable), and a 1-collarable part, 
computed by a {\it 1-collarable pushout} (over the complement clopen subspace of $ A). $ The latter 
pushout is {\it homeomorphic} to a standard homotopy pushout (Thm. II.2.5). Topological spaces, 
collarable cospans and their transversal maps form thus a transversally full weak double subcategory
    \begin{equation}
\CCblc(\Top)  \subset  \CCosp(\Top).
    \label{2.3.6}  \end{equation}

	Notice that a topological map, even if collarable, can not be viewed as a collarable cospan, generally: 
the category $ \Top $ is {\it transversally embedded} in $ \CCblc(\Top), $ sending a map $ f \c X \to 
Y $ to the same 0-map.

\subsection{Pre-collared cospans}\label{2.4}
A {\it pre-collared (topological) cospan} $ \; U = (u; \Um, \Up) \; $ will 
consist of a cospan $ u = (\um \c \Xm \to \Xz \lw \Xp \cc \up) $ of {\it topological embeddings} 
equipped with two {\it pre-collars} $ U^\al ; $ these are maps which extend to the cylinder the 
corresponding maps $ u^\al  $ of $ u$

%
%  diagram for pre-collared cospan
    \begin{equation}
    \begin{array}{c}  \xymatrix  @C=20pt @R=20pt
{
~I\Xm~   \ar[r]^{~\Um}    &  ~\Xz~   &   ~I\Xp~   \ar[l]_{\Up}   &\q&
U^\al  \c IX^\al  \to \Xz,
\\ 
~\Xm~   \ar[u]^{~\dm}   \ar[ru]_{~\um}    &&   ~\Xp~  \ar[u]_{~\dm}   \ar[lu]^{~\up}   &&
u^\al (x)  =  U^\al (x, 0).
}
    \label{2.4.1} \end{array} \end{equation}
\smallskip

	Notice that these pre-collars need not be injective, nor have disjoint images, and we are not (yet) 
assuming the existence of a `collared decomposition'. The {\it underlying cospan} of $ U $ is $ |U| = 
u. $ Faces and degeneracies are defined as follows, consistently with the procedure $ | - |$:
    \begin{equation}
\dda U  =  \dda u  =  X^\al ,  \;\;
e_1(X)  =  (X \to X \lw X;  eX \c IX \to X \lw IX \cc eX).
    \label{2.4.2}  \end{equation}

	Pre-collared cospans form a 1-cubical set, with topological spaces in degree 0.

	Every collarable cospan $ u = (\um \c \Xm \to \Xz \lw \Xp \cc \up) $ underlies some pre-collared 
cospan $ U. $ In fact, if $ u $ is 1-collarable, it admits a collar-cospan (2.3); if $ u $ is 0-collarable, 
i.e., a pair of homeomorphisms, then it admits a {\it trivial pre-collar}, similar to a degenerate one, with 
$ U^\al  = u^\al .eX^\al  \c IX^\al  \to X^\al  \to \Xz. $ In the general case, $ u $ has a collarable decomposition 
(\ref{2.3.3}), and it suffices to take the topological sum of the previous solutions - also because the 
cylinder functor preserves sums.

\subsection{The weak double category of pre-collared cospans}\label{2.5}
Let us suppose we have two pre-collared cospans, 
$ U = (u; U^\al ) $ as above (in (\ref{2.4.1})) and
    \begin{equation}
V  =  (v; \Vm, \Vp),	\qq    v  =  (\vm \c \Ym \to \Yz \lw \Yp \cc \vp),
    \label{2.5.1}  \end{equation}
which are consecutive: $ \Xp = A = \Ym. $ As in the collarable case (Part II), to define their {\it 
concatenation} $ W = U +_1 V, $ we concatenate their underlying cospans getting a cospan $ w = u 
+_1 v $ (of topological embeddings, by \ref{2.1}) and then form the new pre-collars using two of the 
old ones, namely $ \Um $ and $ \Vp $
%
%  constructing the pre-collars
    \begin{equation}
    \begin{array}{c}  \xymatrix  @C=5pt @R=3pt  
{   
&&&&   ~Z^0~    &&&&&
U +_1 V  =  (w; W^\al ),
\\   
&&&   \ar@{--}[rd]   &&  \ar@{--}[ld]    &&&&   w = u +_1 v
\\
~I\Xm~   \ar[rr]^-{\Um}  &&   ~\Xz~ \ar[rruu]^\hm   &&&&
~\Yz~  \ar[lluu]_\hp  &&  ~I\Yp~   \ar[ll]_-{\Vp}
&
\q\;\;  =  (\hm \um, \hp \up),
\\
&&&&&&&&&
W^-   =  \hm \Um,
\\
  ~~\Xm~   \ar[uu]^{\dm}   \ar[rruu]_\um    &&&&
~A~  \ar[lluu]^\up  \ar[rruu]_\vm   &&&&  ~\Yp~  \ar[uu]_{\dm}   \ar[lluu]^\vp
&  W^+   =  \hp \Vp,
}
    \label{2.5.2} \end{array} \end{equation}
	A {\it transversal map} of pre-collared cospans
    \begin{equation}
f  =  (\fm, \fz, \fp) \c U \to V,	\qq   f^t \c X^t  \to \Y^t   \qq   (t \in  \We),
    \label{2.5.3}  \end{equation}
is a triple of maps $ f^t $ which commute with the pre-collars of $ U, V, $ via their cylindrical 
extensions $ F^\al  = If^\al  \c IX^\al  \to IY^\al . $ The {\it underlying transversal map} $ |f| \c |U| \to |V| $ 
has the same components $ f^t.$

	Faces, degeneracies and transpositions of transversal maps are defined in the obvious way, 
consistently with domains and codomains. Finally, the comparison for associativity is provided by the homologous comparison for the 
concatenation of the underlying cospans in the weak double category $ \CCosp(\Top ), $ which is 
easily seen to satisfy the condition of consistence with pre-collars.

	We have thus defined the {\it weak double category} $ \pCCc(\Top ) $ {\it of pre-collared 
cospans} and equipped it with a forgetful double functor
    \begin{equation}
| - | \c \pCCc(\Top ) \to \CCosp(\Top ),
    \label{2.5.4}  \end{equation}
which is the identity on objects and faithful on transversal maps. (Concatenation is strictly preserved, 
since we are using {\it one} choice of distinguished pushouts in $ \Top $ to concatenate cospans, in 
both structures.)

\subsection{Collared cospans}\label{2.6}
Let $ U = (u; U^\al ) $ be a pre-collared cospan, as in (\ref{2.4.1}). We say 
that $ U $ is {\it collared} if $ u $ can be decomposed into a (uniquely determined) binary sum, 
where:
%
% collared decomposition
    \begin{equation}
u  =  u_0 + u_1  =  (\Sum \um_i \c  \Sum \Xm_i \to \Sum \Xz_i \lw \Sum \Xp_i   \cc \Sum \up_i)\q  (i = 0, 1),
    \label{2.6.1}  \end{equation}
    \begin{equation}
U^\al (IX^\al_0)  \subset  X^0_0,   \qq   U^\al (IX^\al_1)  \subset  X^0_1,
    \label{2.6.2}  \end{equation}
so that the restrictions $ (\Um_i, \Up_i) $ form a pre-collar cospan of $ u_i; $ more precisely, we want that:

\Ndt  (i)  $ u_0 = (\um_0, \up_0) $ {\it is a pair of homeomorphisms} and $ U^\al_0 = u^\al_0.eX^\al_0 \c IX^\al_0 \to X^\al_0 \to X^0_0$;

\Ndt   (ii) the pair $ (\Um_1, \Up_1) $ is {\it a 1-collar of} $ u_1 = (\um_1, \up_1), $ i.e. it consists of two disjoint closed embeddings (extending the maps $ u^\al_1) $ and $ U^\al_1(X^\al_1 \ti  [0, 1[) $ is open in $ X^0_1$.

\smallskip

	Collared cospans form the transversally full {\it weak double subcategory} $ \CCc(\Top ) \subset 
\pCCc(\Top ). $ The fact that they are closed under concatenation in the latter is proved as in Part II 
for collarable cospans (Thm. II.2.4).

	We have thus a commutative triangle of double functors
    \begin{equation}
    \begin{array}{c}  \xymatrix  @C=20pt @R=20pt
{
~~\CCc(\Top )~~   \ar@{^{(}->}[r]  \ar[rd]  &    ~~\pCCc(\Top ) ~~   \ar[d]^{| - |}
\\ 
&    ~~\CCosp(\Top )~~
}
    \label{2.6.3} \end{array} \end{equation}
\smallskip
which are transversally faithful (the inclusion is also transversally full).

\section{Pre-collars in cubical degree 2}\label{3}

	We define here the weak 3-cubical category $ 2\pCCc_*(\Top) $ of 2-cubical pre-collared cospans. 
The collared case will be treated directly in unbounded cubical degree, in the next section.

\subsection{Notation for cubical cospans}\label{3.1}
Let  $ \t = (t_1,..., t_n) \in  \We^n $ be a multi-index with 
coordinates $ t_i \in \{- 1, 0, 1\}. $ If $ t_i \neq 0, $ we write
    \begin{equation}
\t^{\sharp i}  =  (t_1,..., t_{i-1}, 0, t_{i+1},..., t_n),
    \label{3.1.1}  \end{equation}
the point of $ \We^n $ obtained by annihilating the $i$-th coordinate.

	A cubical cospan $ u \c \We^n \to \Top $  will be written as follows
    \begin{equation}
u  =  (X(\t), u(i, \t)),  \; \; \;  u(i, \t) \c X(\t) \to X(\t^{\sharp i} )  \; \; \;  (i = 1,..., n; \;  \t \in \We^n;  \;  t_i \neq  0),
    \label{3.1.2}  \end{equation}
where $ u(i, \t) $ is a map in direction $ i, $ as exemplified below for $ n = 1, 2$
%
%  cospan
    \begin{equation}
    \begin{array}{c}  \xymatrix  @C=40pt @R=20pt
{
\q\;\;   ~\Xm~~   \ar[r]^{u(1-)}     &  ~~\Xz~~  &  ~~\Xp~~   \ar[l]_{u(1+)}  & \q &
}
    \label{3.1.3} \end{array} \end{equation}
%
%
%  2-cospan
    \begin{equation}
    \begin{array}{c}  \xymatrix  @C=40pt @R=30pt
{
~X^{--}~   \ar[r]^-{u(1--)}    \ar[d]_{u(2--)}  &  ~X^{0-}~   \ar[d]^{u^{0-}}  &
~X^{+-}~   \ar[l]_-{u(1+-)}  \ar[d]^{u(2+-)} 
\\
~X^{-0}~   \ar[r]^-{u^{-0}}     &  ~X^{00}~  &  ~X^{+0}~   \ar[l]_-{u^{+0}}    & 
\Bu    \ar[r]^1   \ar[d]^2  &
\\  
~X^{-+}~   \ar[r]_-{u(1-+)}    \ar[u]^{u(2-+)}  &  ~X^{0+}~   \ar[u]_{u^{0+}}   &  ~X^{++}~   \ar[l]^-{u(1++)}  \ar[u]_{u(2++)} &&
}
    \label{3.1.4} \end{array} \end{equation}
\smallskip

	Note, in the latter, the simplified notation of the central arrows: $ u^{-0} = u(1, -, 0), $ and so on.

	For a fixed $ i, $ there are $ 2.3^{n-1} $ maps $ u(i, \t) $ in direction $ i. $ Globally, there are $ 
2n.3^{n-1} $ maps in $ u $ and $ n.3^{n-1} $ cospans. Of course, when we speak of a {\it cospan of } 
$u $ we always mean two maps with the same codomain {\it and the same direction}, i.e. a pair $ (u(i, 
\t)), $ with $ t_i = \pm 1 $ (the other coordinates being fixed, as well as $ i$).

	An $n$-cubical cospan $ u $ can also be viewed as a cospan (in direction 1) of $(n-1)$-cubical 
cospans (in directions $ 2,..., n$)
    \begin{equation}
    \begin{array}{ll}
u \c \We \to \Top^{\We^{n-1}},	&   u  =  (\um \c \Xm \to \Xz \lw \Xp \cc \up),
\\[5pt]
X^t   =  (X(t, \t),  u(i, t, \t)),   &   u(i, t, \t) \c  X(t, \t) \to X(t, \t^{\sharp i} ),
    \label{3.1.5}\end{array}
    \end{equation}

\ndt   where $ i = 1,..., n-1 $ and $ \t = (t_2,..., t_n) \in \We^{n-1}$, $ t_i \neq  0. $ (Other presentations of $ u $ can 
be obtained from the transposed cospans $ us_i, $ using the transposition symmetries $ s_i \c \We^n 
\to \We^n, $ cf. \ref{1.4}.)

\subsection{Square cospans and pre-collars}\label{3.2}
   Let 
$ u \c \We^2 \to \Top $ be a 2-cubical topological cospan, with the previous notation (in (\ref{3.1.4})).

	As considered above (in degree $ n$), $ u $ will also be viewed as a cospan (in direction 1) of 
cospans (in direction 2), as follows
    \begin{equation}
u \c \We \to \Top^\We,	  \qq    u  =  (\um \c \Xm \to \Xz \lw \Xp \cc \up),
    \label{3.2.1}  \end{equation}
    \begin{equation}
    \begin{array}{ll}
X^t   =  (u(2, t, -) \c  X^{t-} \to X^{t0} \lw X^{t+} \cc u(2, t, +)),
\\[3pt]
u^\al   =  (u(1, \al , -),  u(1, \al , 0),  u(1, \al , +))	\q  (t \in \We,  \; \al  = \pm ).
    \label{3.2.2}\end{array}
    \end{equation}

	The symmetric presentation of $ u $ can be obtained as above from the transposed cospan $ s_1(u) = 
us \c \We^2 \to \Top, $ using the transposition symmetry $ s \c \We^2 \to \We^2 $ (\ref{1.4}).

	Let us assume that all maps of $ u $ are topological embeddings. A {\it family of pre-collars of $ u 
$ in direction 1} consists of six maps
    \begin{equation}
U(1, \al , t) \c IX^{\al t} \to X^{0t}	   \qq    (\al  = \pm ,  \; t \in \We),
    \label{3.2.3}  \end{equation}
%
%
%
%  pre-collars in direction 1
    \begin{equation}
    \begin{array}{c}  \xymatrix  @C=30pt @R=30pt
{
~IX^{--}~   \ar[r]^-{U(1--)}    \ar[d]_{Iu(2--)}  &  ~X^{0-}~   \ar[d]^{u^{0-}}  &
~IX^{+-}~   \ar[l]_-{U(1+-)}  \ar[d]^{Iu(2+-)} 
\\
~IX^{-0}~   \ar[r]^-{U^{-0}}     &  ~X^{00}~  &  ~IX^{+0}~   \ar[l]_-{U^{+0}}    && 
\Bu    \ar[r]^1   \ar[d]^2  &
\\  
~IX^{-+}~   \ar[r]_-{U(1-+)}    \ar[u]^{Iu(2-+)}  &  ~X^{0+}~   \ar[u]_{u^{0+}}   &
~IX^{++}~   \ar[l]^-{U(1++)}  \ar[u]_{Iu(2++)} &&
}
    \label{3.2.4} \end{array} \end{equation}

\Ndt   so that:

\Ndt   (i) each of them is a pre-collar of the corresponding map $ u(1, \al , t), $ i.e. 
$ u(1, \al , t) = U(1, \al , t)(-, 0) $ (\ref{2.4});

\Ndt   (ii) the diagram (\ref{3.2.4}) commutes and its four squares are pullbacks.

\smallskip
  	Symmetrically, a {\it family of pre-collars of $ u $ in direction 2} amounts to the previous notion 
for the transposed 2-cospan $ s_1(u) = us. $ It consists thus of six maps
    \begin{equation}
U(2, t, \al ) \c IX^{t\al}  \to X^{t0}	\qq  (\al  = \pm , \; t \in \We),
    \label{3.2.5}  \end{equation}
which satisfy symmetric conditions. In particular, they form four pullbacks 
%
%  pre-collars in direction 2
    \begin{equation}
    \begin{array}{c}  \xymatrix  @C=40pt @R=20pt
{
~IX^{--}~   \ar[r]^-{Iu(1--)}    \ar[d]_{U(2--)}  &  ~IX^{0-}~   \ar[d]^{U^{0-}}  &
~IX^{+-}~   \ar[l]_-{Iu(1+-)}  \ar[d]^{U(2+-)} 
\\
~X^{-0}~   \ar[r]^-{u^{-0}}     &  ~X^{00}~  &  ~X^{+0}~   \ar[l]_-{u^{+0}}    & 
\Bu    \ar[r]^1   \ar[d]^2  &
\\  
~IX^{-+}~   \ar[r]_-{Iu(1-+)}    \ar[u]^{U(2-+)}  &  ~IX^{0+}~   \ar[u]_{U^{0+}}   &
~IX^{++}~   \ar[l]^-{Iu(1++)}  \ar[u]_{U(2++)} &&
}
    \label{3.2.6} \end{array} \end{equation}

 	A {\it pre-collared topological 2-cospan} $ U = (u; U(1, \al , t), U(2, t', \be )) $ will be a 2-cospan 
$ u \c \We^2 \to \Top $ of topological embeddings, equipped with a {\it family of pre-collars}, i.e. a 
pair of such families in both directions. Its {\it underlying 2-cospan} is $ |U| = u$.

	One can note that, because of the pullback condition (ii), all the pre-collars are determined by the 
{\it collar-cross}, consisting of the 4 {\it central collars}, i.e. those which reach the central object

%
%  the collar cross
    \begin{equation}
    \begin{array}{c}  \xymatrix  @C=30pt @R=25pt
{
&  ~~IX^{0-}~   \ar[d]^{U^{0-}}
\\
~IX^{-0}~   \ar[r]^-{U^{-0}}     &  ~X^{00}~  &  ~IX^{+0}~   \ar[l]_-{U^{+0}}    && 
\Bu    \ar[r]^1   \ar[d]^2  &
\\  
&  ~~IX^{0+}~   \ar[u]_{U^{0+}}   &&&
}
    \label{3.2.7} \end{array} \end{equation}

\subsection{Faces and degeneracies}\label{3.3}
	The faces of the pre-collared 2-cospan $ U $ are the following pre-collared 1-cospans
%
%  faces
    \begin{equation}
    \begin{array}{ll}
 \dda_1 U  =  ( \dda_1 u;  U(2,\al ,-) \c IX^{\al - } \to X^{\al 0} \lw IX^{\al +} \cc U(2,\al ,+)),
\\[4pt]
\ddb_2U  =  (\ddb_2u;  U(1,-,\be ) \c IX^{-\be}  \to X^{0\be}  \lw IX^{+\be}  \cc U(1,+,\be )).
    \label{3.3.1}\end{array}
    \end{equation}

	For instance, the two faces $  \dda_1 U $ can be pictured as follows:

\begin{equation} \begin{array}{c}
\xy <.5mm, 0mm>:
% dummy pts below, above
(30,0) *{}; (30,115) *{};
% labels
(20,5) *{(\ddm_1 U)}; (130,5) *{(\ddp_1 U)};
(0,25) *{IX^{-+}};   (150,25) *{IX^{++}}; 
(0,65) *{X^{-0}};    (150,65) *{X^{+0}}; 
(0,105) *{IX^{--}}; (150,105) *{IX^{+-}}; 
(75,65) *{\sst {X^{00}}};  (77,25) *{\sst {IX^{0+}}}; (77,105) *{\sst {IX^{0-}}};
% labels of arrows
(7,40) *{\sst {U(-+)}}; (83,40) *{\sst {U^{0+}}}; (143,40) *{\sst {U(++)}}; 
(7,90) *{\sst {U(--)}}; (83,90) *{\sst {U^{0-}}}; (143,90) *{\sst {U(+-)}}; 
% 6 segments & thick parts
(20,20); (20,30) **@{-}; (20.4,20); (20.4,30) **@{-}; 
   (130,20); (130,30) **@{-}; (130.4,20); (130.4,30) **@{-}; 
(20,50); (20,80) **@{-}; (20.4,50); (20.4,60) **@{-}; (20.4,70); (20.4,80) **@{-}; 
   (130,50); (130,80) **@{-}; (130.4,50); (130.4,60) **@{-}; (130.4,70); (130.4,80) **@{-}; 
(20,100); (20,110) **@{-}; (20.4,100); (20.4,110) **@{-}; 
   (130,100); (130,110) **@{-}; (130.4,100); (130.4,110) **@{-};
% 3 rectangles
@i @={(60,20), (90,20), (90,30), (60, 30)},
s0="prev"  @@{;"prev";**@{-}="prev"};
@i @={(60,50), (90,50), (90,80), (60, 80)},
s0="prev"  @@{;"prev";**@{-}="prev"};
@i @={(60,100), (90,100), (90,110), (60, 110)},
s0="prev"  @@{;"prev";**@{-}="prev"};
% thick parts of rectangles
(60.4,20); (60.4,30) **@{-}; (89.6,20); (89.6,30) **@{-}; 
(60.4,50); (60.4,60) **@{-}; (89.6,50); (89.6,60) **@{-}; 
(60.4,70); (60.4,80) **@{-}; (89.6,70); (89.6,80) **@{-}; 
(60.4,100); (60.4,110) **@{-}; (89.6,100); (89.6,110) **@{-}; 
% horizontal arrows
\POS(45,25) \ar-(11,0);  \POS(105,25) \ar+(11,0);
\POS(45,65) \ar-(11,0);  \POS(105,65) \ar+(11,0);
\POS(45,105) \ar-(11,0);  \POS(105,105) \ar+(11,0);
% vertical arrows
\POS(20,45) \ar-(0,10);  \POS(75,45) \ar-(0,10);  \POS(130,45) \ar-(0,10);
\POS(20,85) \ar+(0,10);  \POS(75,85) \ar+(0,10);  \POS(130,85) \ar+(0,10);
\endxy
\label{3.3.2} \end{array} \end{equation}

	The degeneracies $ e_1U$, $ e_2U $ of a pre-collared 1-cospan $ U = (u; U^\al ) $ are defined as 
follows
%
%  degeneracies
    \begin{equation}
e_1U  =  (e_1u; eX^t , U^\be ),	\qq    e_2U  =  (e_2u; U^\al , eX^t ).
    \label{3.3.3}  \end{equation}

	In particular, $ e_1U $ has pre-collars forming the following two diagrams of pullbacks
%
%  pre-collars of a degenerate cospan
    \begin{equation}
    \begin{array}{c}  \xymatrix  @C=10pt @R=20pt
{
~IX^-~   \ar[r]^e    \ar[d]_{Iu^-}  &  ~X^-~   \ar[d]^{u^-}  &
~IX^-~   \ar[l]_e  \ar[d]^{Iu^-}  &&
~IX^-~   \ar[r]^-{\id}    \ar[d]_{U^-}  &  ~IX^-~   \ar[d]^{U^-}  &
~IX^-~   \ar[l]_-{\id}  \ar[d]^{U^-}  
\\
~IX^0~   \ar[r]^e     &  ~X^0~  &  ~IX^0~   \ar[l]_e   &&
~X^0~   \ar[r]^-{\id}     &  ~X^0~  &  ~X^0~   \ar[l]_-{\id}   
&   \Bu    \ar[r]^1   \ar[d]^2  &
\\  
~IX^+~   \ar[r]_e    \ar[u]^{Iu^+}  &  ~X^+~   \ar[u]_{u^+}   &
~IX^+~   \ar[l]^e  \ar[u]_{Iu^+} &&
~IX^+~   \ar[r]_-{\id}    \ar[u]^{U^+}  &  ~IX^+~   \ar[u]_{U^+}   &
~IX^+~   \ar[l]^-{\id}  \ar[u]_{U^+} &
}
    \label{3.3.4} \end{array} \end{equation}

\subsection{Concatenating pre-collared 2-cospans}\label{3.4}

\begin{thm}
Pre-collared 2-cospans have well-defined 
concatenation in both directions, consistently with their underlying 2-cospans.
\end{thm}
\begin{proof}
 It suffices to prove two points:
 
 \Ndt   (A) Families of pre-collars in direction 1 can be concatenated in direction 1,
 
\Ndt   (B) Families of pre-collars in direction 2 can be concatenated in direction 1.

\smallskip

	Point (A) is proved as in Thm. II.2.4, working in $ \Top^\We $ instead of $ \Top $ (cf. 
(\ref{3.2.1})). We only have to check that the new squares built by pushouts are indeed pullbacks. 
This is done in the following lemma, \ref{3.5}.

	For point (B), let us assume that $ u, v $ are equipped with pre-collars in direction 2, and prove that 
their concatenation $ w = u +_1 v $ can be canonically so equipped. Below, `in direction 2' is 
understood most of the time.

	Consider the cospan $ (h; H^-, H^+) = \ddp_1 u = \ddm_1 v $ along which the concatenation is computed.

	The 1-faces of $ w $ (whose maps are in direction 2) belong to $ u $ or $ v$  
    \begin{equation}
w(2, -, \al )  =  u(2, -, \al ),	 \qq   w(2, +, \al )  =  v(2, +, \al ),
    \label{3.4.1}  \end{equation}
and are already equipped with pre-collars $ (U(2, -, \al)) $ and $ (V(2, +, \al)).$

	Furthermore, we have a central cospan 
$ (w^{0-}, w^{0+}) = (w(2, 0, -), w(2, 0, +)) $ which is computed in 
the left diagram below, via three pushouts
    \begin{equation}
    \begin{array}{c}  \xymatrix  @C=3pt @R=10pt
{
& ~Y^{0-}~   \ar[rrr]^{y^-}    \ar@{-->}[dd]^<<{v^{0-}}  &&&  ~Z^{0-}~   \ar[dd]^{w^{0-}}
&&&
 ~IY^{0-}~   \ar[rrr]^{Iy^-}    \ar@{-->}[dd]^<<{V^{0-}}  &&&  ~IZ^{0-}~  \ar[dd]^{W^{0-}}
\\
~A^-~   \ar[rrr]^{~~~~u^-}   \ar[ru]^-{v^-}    \ar[dd]_{h^-}  &&&  
~X^{0-}~      \ar[ru]^-{x^-}  \ar[dd]^<<{u^{0-}}
&&&
~IA^-~   \ar[rrr]^{~~~~Iu^-}   \ar[ru]^-{Iv^-}    \ar[dd]_{H^-}  &&&  
~IX^{0-}~      \ar[ru]^-{Ix^-}  \ar[dd]^<<{U^{0-}}
\\  
& ~Y^{00}~   \ar@{-->}[rrr]^{y^0~~~~~}    &&&  ~Z^{00}~   
&&&
 ~Y^{00}~   \ar@{-->}[rrr]^{y^0~~~~~}    &&&  ~Z^{00}~   
\\
~A^0~   \ar[rrr]^-{~~~~u^0}    \ar@{-->}[ru]^-{v^0}  &&&  ~X^{00}~   \ar[ru]_-{x^0}
&&&
~A^0~   \ar[rrr]^-{~~~~u^0}    \ar@{-->}[ru]^-{v^0}  &&&  ~X^{00}~   \ar[ru]_-{x^0}
\\
& ~Y^{0+}~   \ar@{-->}[rrr]^{y^+~~~~~}    \ar@{-->}[uu]^<<{v^{0+}}  &&&  ~Z^{0+}~   \ar[uu]_{w^{0+}}
&&&
 ~IY^{0+}~   \ar@{-->}[rrr]^{Iy^+~~~~~}    \ar@{-->}[uu]^<<{V^{0+}}  &&&  ~IZ^{0+}~   \ar[uu]_{W^{0+}}
\\
~A^+~   \ar[rrr]_-{u^+}    \ar@{-->}[ru]^-{v^+}  \ar[uu]^{h^+}  &&&  
~X^{0+}~   \ar[uu]_>>>{u^{0+}}    \ar[ru]_-{x^+}
&&&
~IA^+~   \ar[rrr]_-{Iu^+}    \ar@{-->}[ru]^-{Iv^+}  \ar[uu]^{H^+}  &&&  
~IX^{0+}~   \ar[uu]_>>>{U^{0+}}    \ar[ru]_-{Ix^+}
}
    \label{3.4.2} \end{array} \end{equation}

 	Here, the notation of some maps is simplified:
$$  
		u^\al   =  u(1, \al , +),	\q	v^\al   =  v(1, \al , -),  \q	h^\al   =  u(2, +, \al )  =  v(2, -, \al ),
$$
and similarly for their collars, denoted by the corresponding capital letters.

	Now, for this last cospan $ (w^{0-}, w^{0+}). $ we have to construct pre-collars, 
consistently with the 
previous collars $ (U(2, -, -), U(2, -, +)) $ and $ (V(2, +, -), V(2, +, +)).$

	The cylinder functor $ I \c \Top \to \Top $ preserves pushouts, as a left adjoint. Therefore, the 
right diagram above allows us to define - consistently - the vertical arrows $ (W^{0-}, W^{0+}), $ on the 
basis of the consistent collars appearing in the other vertical arrows. They are also topological 
embeddings, and the new squares built in the right diagram above are pullbacks, again by 
\ref{3.5}.
\end{proof}	

\subsection{A diagrammatic lemma}\label{3.5}
We end this section with a diagrammatic property, which has already been used in the proof of the previous theorem. (In that case, the bottom square of the diagram below is a pushout of embeddings, which is also a pullback.)

\begin{lemma}
Let us suppose we have, in $ \Top$, a commutative cubical diagram of topological embeddings:
    \begin{equation}
    \begin{array}{c}  \xymatrix  @C=7pt @R=10pt
{
& ~Y~   \ar[rrr]   \ar@{-->}[dd]  &&&  ~Z~   \ar[dd]
\\
~A~   \ar[rrr]   \ar[ru]    \ar[dd]  &&&  ~X~      \ar[ru]  \ar[dd]
\\  
& ~Y'~   \ar@{-->}[rrr]   &&&  ~Z'~   
\\
~A'~   \ar[rrr]    \ar@{-->}[ru]  &&&  ~X'~   \ar[ru]
}
    \label{3.5.1} \end{array} \end{equation}

	If the front and bottom squares are pullbacks, and the top square is a pushout, then also the back 
square is a pullback.
\end{lemma}
\begin{proof}
Since our maps are topological embeddings, we can forget topologies and work in the category of 
sets. The pasting of the front and bottom squares is a pullback. Therefore, the pasting of the top and 
back squares is also a pullback (the same). We have thus a commutative diagram of sets
    \begin{equation}
    \begin{array}{c}  \xymatrix  @C=10pt @R=10pt
{
~A~   \ar[rr]   \ar[dd]  &&  ~Y~   \ar[rr]   \ar[dd]   &&  ~Y'~    \ar[dd]
\\ 
&    \ar@{--}[r]   \ar@{--}[d] &
\\
~X~   \ar[rr]    &&   ~Z~   \ar[rr]   &&  ~Z'~  
}
    \label{3.5.2} \end{array} \end{equation}
where the left square is a pushout and the outer rectangle a pullback. Knowing that all maps are 
injective, it is quite easy to check, on elements, that the right square is also a pullback.
\end{proof}

\section{Collared cubical cospans}\label{4}

	We extend the previous constructions to unbounded cubical degree, constructing the symmetric weak 
cubical category of pre-collared cubical cospans $ \pCCc_*(\Top) $ and its weak cubical subcategory $ 
\CCc_*(\Top) $ of collared cubical cospans.
	
\subsection{Pre-collared cubical cospans}\label{4.1}
	We extend now the previous definitions (\ref{2.4}, \ref{3.2}) to 
higher cubical degree, using the notation for cubical cospans introduced in \ref{3.1}.

 	A {\it pre-collared n-cubical cospan} $ \; U = (u; U(i, \t)) \; $ consists of an $n$-cubical cospan $ u = 
|U| $ of topological embeddings, equipped with a family of $ 2n.3^{n-1} $ {\it pre-collars} $ U(i, \t) $ 
of its maps
    \begin{equation}
    \begin{array}{ll}
u  =  |U|  =  (X(\t), u(i, \t)),   \q&   U(i, \t) \c IX(\t) \to X(\t^{\sharp i} ),
\\[3pt]
u(i, \t)(x)  =  U(i, \t)(x, 0)  & 
(i = 1,..., n,  \; \t = (t_1,..., t_n) \in \We^n,  \; t_i \neq   0).
    \label{4.1.1}\end{array}
    \end{equation}

	Moreover, the following squares must (commute and) be {\it pullbacks} 
    \begin{equation}
    \begin{array}{c}  \xymatrix  @C=20pt @R=5pt
{
~IX(\t)~   \ar[rr]^{Iu(i, \t)}   \ar[dd]_{U(j, \t)}  &&  
~IX(\t^{\sharp i} )~    \ar[dd]^{U(j, \t^{\sharp i} )}
\\ 
 \ar@{--}[r]  &  \ar@{--}[u] && \qq   (i \neq  j, \; t_i \neq   0, \; t_j \neq   0).
\\
~X(\t^{\sharp i} )~   \ar[rr]_{u(j, \t^{\sharp j})}    &&   ~X(\t^{\sharp i \sharp j})~   
}
    \label{4.1.2} \end{array} \end{equation}

	Faces and degeneracies are defined as follows, forming a symmetric cubical set (consistently with 
the forgetful procedure $ | - |), $ where $ i^{\sharp j} $ is $ i $ if $ i < j, $ and $ i + 1 $ if $ i \ge j$
    \begin{equation}
    \begin{array}{ll}
\dda_j U  =  (\dda_j u, U(i^{\sharp j}, \dda_j \t))		&	(i  =  1,..., n-1;  \; \t \in \We^{n-1}),
\\[3pt]
e_j(U)  =  (e_ju, e_j(U)(i, \t)),
\\[3pt]
e_j(U)(i, \t)  =  U(i, e_j\t) \;\;\;  (i \neq  j),	\q &   e_i(U)(i, \t)  =  eX(\t) \c IX(\t) \to X(\t),
    \label{4.1.3}\end{array}
    \end{equation}
($ \dda_j \t$, $ e_j\t $ are defined in (\ref{1.4.5}).)

\subsection{A weak cubical category}\label{4.2}
	We can now form the symmetric weak cubical category $ 
\pCCc_*(\Top) $ of {\it pre-collared cubical cospans}.

	The operation of $i$-concatenation of pre-collared $n$-cubical cospans is defined as in \ref{3.4} 
for $ n = 2$: point (A) and (B) specify, respectively, how to construct the new collars in direction $ i $ 
and $ j \neq  i.$

	A {\it transversal map} of pre-collared $n$-cubical cospans, also called an $(n+1)$-cell,
    \begin{equation}
f \c U \to V,    \qq    U  =  (u, U(j, \t)),   \q   V  =  (v, V(j, \t)),
    \label{4.2.1}  \end{equation}
is a transversal map $ |f| \c |U| \to |V| $ of the underlying $n$-cubical cospans which commutes with the collars, via the cylindrical extensions
    \begin{equation}
F(\t)  =  If(\t) \c IX(\t) \to IY(\t)   \qq   (\t \in \We^n).
    \label{4.2.2}  \end{equation}

	Faces, degeneracies and transpositions of transversal maps are defined in the obvious way, 
consistently with domains and codomains. The comparisons for associativity and interchange derive 
from the comparisons of $ \CCosp_*(\Top )$.

	We have also defined a cubical forgetful functor
    \begin{equation}
| - | \c  \pCCc_*(\Top) \to \CCosp_*(\Top),
    \label{4.2.3}  \end{equation}
which is {\it transversally faithful}: given two transversal maps $ f, g \c U \to V, $ the condition $ |f| = 
|g| $ implies $ f = g.$	

\subsection{Cubical collared cospans}\label{4.3}
	 Let $ U = (u, U(i, \t)) $ be a pre-collared $n$-cubical cospans.

	We say that $ U $ is {\it collared in direction  i} $ (= 1,..., n) $ if the underlying cubical cospan $ 
u $ can be decomposed into a binary sum, coherently with its pre-collars
    \begin{equation}
    \begin{array}{ll}
u(i, \t)  =  u_0(i, \t) + u_1(i, \t),
\\[3pt]
u_j(i, \t) \c X_j(\t)  \to  X_j(\t^{\sharp i} )   \q   (j = 0, 1;  \;  t_i \neq  0),
    \label{4.3.1}\end{array}
    \end{equation}
    \begin{equation}
U(i, \t)(IX_0(\t))  \subset  X_0(\t^{\sharp i} ),    \q  U(i, \t)(IX_1(\t))  \subset  X_1(\t^{\sharp i} ),
    \label{4.3.2}  \end{equation}
so that the restrictions 
    \begin{equation}
U_0(i, \t) \c IX_0(\t) \to X_0(\t^{\sharp i} ),	\q  U_1(i, \t) \c IX_1(\t) \to X_1(\t^{\sharp i} ),
    \label{4.3.3}  \end{equation}
yield a collared decomposition of each ordinary cospan 
$ (u(i, \t); t_i = \pm ). $

 By \ref{2.6}, this means that:

\Ndt   (i) each cospan $ (u_0(i, \t); \; t_i = \pm ) $ {\it is a pair of homeomorphisms} with trivial collars $ U_0(i, \t) 
= u_0(i, \t).eX_0(\t) \c IX_0(\t) \to X_0(\t) \to X_0(\t^{\sharp i} )$;

\Ndt   (ii) the pair $ (U_1(i, \t); t_i = \pm ) $ is {\it a 1-collar of} $ (u_1(i, \t); \; t_i = \pm ), $ i.e. it consists of two 
disjoint closed embeddings (extending the maps $ u_1(i, \t)) $ and $ U_1(i, \t)(X_1(\t) \ti  [0, 1[) $ is 
open in $ X_1(\t^{\sharp i} )$.

	We say that $ U $ is {\it collared} if it is collared in every direction.
	 
\subsection{Concatenating collared cubical cospans} \label{4.4}
	 
\begin{thm}
Collared cubical co-spans are stable under 
concatenation in all directions, in $ \pCCc_*(\Top).$
\end{thm}
\begin{proof}
Since transpositions permute directions, it suffices to consider a concatenation $ W = U +_1 V 
$ in direction 1 and prove that:

\Ndt  (A) The new cospans in direction 1 produced by 1-concatenation are collared in direction 1,

\Ndt  (B) The same are collared in direction 2.

\smallskip

	Point (A) is proved as in Thm. II.2.4, working in $ \Top^{\We^{n-1}} $ instead of $ \Top.$

	For point (B), we use the notation of \ref{3.4}, {\it and we only write the indices in 
directions} 1, 2 (which amounts to working in $ \Top^{\We^{n-2}}). $ Collared will mean: collared in 
direction 2.

	Consider the cospan $ H = (h; H^-, H^+) = \ddp_1U = \ddm_1V $ 
along which the concatenation is computed, with collar-pair
    \begin{equation}
H^- \c IA^- \to A^0 \lw IA^+ \cc H^+,	\qq   H  =  H_0 + H_1.
    \label{4.4.1}  \end{equation}

\Ndt   (B$'$) First, let us consider the case $ h \neq  e_1(\es)$, so that in the collared decomposition 
$ h = h_0 + h_1$ both components are non trivial.  Accordingly, we can split the concatenation $ w $ into the sum
    \begin{equation}
w  =  w_0 + w_1  =  (u_0 +_1 v_0) + (u_1 +_1 v_1),
    \label{4.4.2}  \end{equation}
of the concatenations of the 0- and 1-collared parts (in direction 2). Plainly, the first component is 0-collared. 
We can forget it, assuming that $ u, v $ are 1-collared and prove that $ w = u +_1 v $ is also.

	We proceed now as in the proof of 3.4, point (B). The 1-faces of  $w$  (see (\ref{3.4.1})) are already 1-collared, and we construct the new pre-collars  $ (W^{0-}, W^{0+})$  as in (\ref{3.4.2}). We end point (B$'$) proving that the latter are indeed 1-collars.
	
	Some lengthy calculations are needed to show that they are disjoint; essentially, this depends not only on the fact that the old collars are disjoint, but also on the 
pullback-hypothesis (\ref{4.1.2}). First, since the top and bottom squares are pushouts
    \begin{equation}
\Im(W^{0\al})  =  \Im(W^{0\al}.Ix^\al ) \cup Im(W^{0\al}.Iy^\al )  =  x^0(\Im U^{0\al}) \cup y^0(\Im V^{0\al}).    \label{4.4.3}  \end{equation}
	We have thus to consider four intersections; by symmetry, it is sufficient to prove that:
    \begin{equation}
x^0(\Im U^{0-}) \cap x^0(\Im U^{0+})  =  \es, \qq  
x^0(\Im U^{0-}) \cap y^0(\Im V^{0+})  =  \es.
    \label{4.4.4}  \end{equation}

	The first fact is obvious, because $ x^0 $ is injective and the collars $ U^{0\al}  $ are disjoint. For the 
second, we will use the fact that the square of $ (u^0, v^0, x^0, y^0) $ is a pullback (as a pushout of 
embeddings) and the square around $ H^- $ is also (by hypothesis). Therefore:
$$
x^0(\Im U^{0-}) \cap y^0(\Im V^{0+})  \;\;  \subset   \;\;  x^0(\Im U^{0-}) \cap \Im y^0
$$   $$
= x^0(\Im U^{0-} \cap \Im u^0)  =  x^0(\Im (u^0H^-))  =  \Im (x^0 u^0 H^-),
$$
and symmetrically
$$
x^0(\Im U^{0-}) \cap y^0(\Im V^{0+})  \subset  \Im (y^0v^0H^+).
$$
	We conclude noting that $ x^0 u^0 = y^0 v^0, $ and that the collars $ H^\al  $ are disjoint.

\Ndt   (B$''$) Finally, let us examine the degenerate case $ h = e_1(\es). $ If, {\it in direction 2}, the cospans $ u 
$ and $ v $ are both 0-collared or both 1-collared, we come back to the previous argument. But here 
one of them can be 0-collared, say $ u, $ and the other 1-collared. Then the cospan $ (w^{0-}, w^{0+})$  is the sum
    \begin{equation}
    \begin{array}{c}  \xymatrix  @C=7pt @R=10pt
{
& ~Y^{0-}~   \ar[rrr]^{y^-}    \ar@{-->}[dd]^<<{v^{0-}}  &&&  ~Z^{0-}~   \ar[dd]^{w^{0-}}
\\
~\es~   \ar[rrr]   \ar[ru]    \ar[dd]  &&&  
~X^{0-}~      \ar[ru]_-{x^-}  \ar[dd]^<<{u^{0-}}
\\  
& ~Y^{00}~   \ar@{-->}[rrr]^{y^0~~~~~}    &&&  ~Z^{00}~   
\\
~\es~   \ar[rrr]    \ar@{-->}[ru]  &&&  ~X^{00}~   \ar[ru]_-{x^0}
\\
& ~Y^{0+}~   \ar@{-->}[rrr]^{y^+~~~~~}    \ar@{-->}[uu]^<<{v^{0+}}  &&&  ~Z^{0+}~   \ar[uu]_{w^{0+}}
\\
~\es~   \ar[rrr]    \ar@{-->}[ru]  \ar[uu]  &&&  
~X^{0+}~   \ar[uu]_>>>{u^{0+}}    \ar[ru]_-{x^+}
}
    \label{4.4.5} \end{array} \end{equation}
\smallskip

\ndt  of a 0-collared component $ (u^{0-}, u^{0+}) $ and a 1-collared component 
$ (v^{0-}, v^{0+}), $ which means that it is collared.
\end{proof} 

\subsection{The structure of collared cubical cospans}

We can now define the {\it transversally full} weak cubical subcategory of {\it 
collared cubical cospans}
    \begin{equation}
\CCc_*(\Top)  \;  \subset  \;  \pCCc_*(\Top).
    \label{4.5.1}  \end{equation}

	Its $n$-cubes are the collared $n$-cospans, as defined above. Its transversal maps are all the 
natural transformations $ f \c u \to u' \c \We^n \to \Top $ between collared $n$-cospans. 

These data 
are plainly closed under faces and degeneracies. They are also closed under concatenation in any 
cubical direction, as proved in \ref{4.4}. The comparisons of the weak structure 
are inherited from $ \pCCc_*(\Top), $ since they are (invertible special) transversal maps - and we are 
taking all of them between collared cubes.

\section{Cylindrical degeneracies and cylindrical concatenation}\label{5}
	We define a new framework, $ \CCOSP_*(\Top), $ from which we will abstract the notion of a 
symmetric {\it quasi} cubical category (Section 7).

\subsection{Comments}\label{5.1}
	We come back to arbitrary topological cospans and begin a different construction, 
which does not use collars but {\it cylindrical} degeneracies and {\it cylindrical} concatenations (by 
homotopy pushouts). Notice that the latter grants, by itself, a homotopy-invariant concatenation.

	With respect to the symmetric weak cubical category $ \CCosp_*(\Top), $ the new structure 
$\CCOSP_*(\Top) $ differs with respect to degeneracies, concatenations and comparisons, and 
satisfies weaker axioms: degeneracies behave now in a weaker way and - for instance - are just {\it lax} 
identities. We get thus a symmetric {\it quasi} cubical category, as defined in Section 7. (The 
importance of weak units in homotopy theory is discussed in \cite{Ko} and references therein.)

	With respect to $ \CCc_*(\Top), $ the new framework is simpler and can easily be restricted to {\it 
manifolds with faces} and their cobordisms (see the next section). This is actually the main reason for 
using here cylindrical degeneracies instead of the ordinary ones - which would also give lax identities, 
with respect to cylindrical concatenation.

\subsection{Cylindrical degeneracies}\label{5.2}
	Let us come back to the symmetric weak cubical category $ 
\CCosp_*(\Top), $ recalled in Section 1.

	After the ordinary degeneracy $ e_1(X) = (\id X \c X \to IX \lw X \cc \id  X) $ of a space $ X, $ we 
also have a {\it cylindrical degeneracy}:
    \begin{equation}
E_1(X)  =  (\dm \c X \to IX \lw X \cc \dpp),	\q  \dm(x)  =  (x, 0),  \;\;    \dpp(x)  =  (x, 1).
    \label{5.2.1}  \end{equation}
	One degree up, a cospan $ u = (\um \c \Xm \to \Xz \lw \Xp \cc \up) $ has two cylindrical 
degeneracies, $ E_1(u) $ and $ E_2(u) = E_1(u).s $ (with $ \dda_i E_i(u) = u)$

\bigskip
 \xymatrix  @C=18pt @R=25pt
{
& ~X^-~   \ar[r]^-{d^-X^-}    \ar[d]_{u^-}  &  ~IX^-~   \ar[d]^{Iu^-}  &
~X^-~   \ar[l]_-{d^+X^-}  \ar[d]^{u^-}  &
~X^0~   \ar[r]^-{u^-}    \ar[d]_{d^-X^-}  &  ~X^0~   \ar[d]^{d^-X^0}  &
~X^+~   \ar[l]_-{~u^+}  \ar[d]^{d^-X^+}  
\\
& ~X^0~   \ar[r]^-{d^-X^0}     &  ~IX^0~  &  ~X^0~   \ar[l]_-{d^+X^0}   &
~IX^-~   \ar[r]^-{Iu^-}     &  ~IX^0~  &  ~IX^+~   \ar[l]_-{Iu^+}   
&   \Bu    \ar[r]^1   \ar[d]^2  &
\\  
& ~X^+~   \ar[r]_-{d^-X^+}    \ar[u]^{u^+}  &  ~IX^+~   \ar[u]_{Iu^+}   &
~X^+~   \ar[l]^-{d^+X^+}   \ar[u]_{~u^+} &
~X^-~   \ar[r]_-{u^-}    \ar[u]^{d^+X^-}   &  ~X^0~   \ar[u]_{d^+X^0}    &
~X^+~   \ar[l]^-{~u^+}  \ar[u]_{d^+X^+}  &
}
\smallskip
    \begin{equation}
\qq    E_1(u)   \qqq \q  E_2(u) =  E_1(u).s.   \qq
    \label{5.2.2}  \end{equation}

	While ordinary degeneracies satisfy the cubical relation $ e_1e_1 = e_2e_1 $ (\ref{1.1}), the 2-cospans
 $ E_1E_1(X) $ and $ E_2E_1(X) = E_1E_1(X).s $ are different, since $ E_1E_1(X) $ is not symmetric
	
%
%  diagram for  E_1E_1(X)
    \begin{equation}
    \begin{array}{c}  \xymatrix  @C=27pt @R=25pt
{
~X~   \ar[r]^-{d^-}    \ar[d]_{d^-}  &  ~IX~   \ar[d]^{Id^-}  &
~X~   \ar[l]_-{d^+}  \ar[d]^{d^-}    
\\
~IX~   \ar[r]^-{d^-I}     &  ~I^2X~  &  ~IX~   \ar[l]_-{d^+I}      
&&   \Bu    \ar[r]^1   \ar[d]^2  &
\\  
~X~   \ar[r]_-{d^-}    \ar[u]^{d^+}  &  ~IX~   \ar[u]_{Id^+}   &
~X~   \ar[l]^-{d^+}   \ar[u]_{~d^+}   &&
}
    \label{5.2.3} \end{array} \end{equation}
\smallskip

	However, the transposition symmetry $ s \c I^2 \to I^2 $ induces an invertible special comparison
    \begin{equation}
\si _1X \c E_1E_1(X) \to E_2E_1(X)	\qq   (E_2E_1 = s_1E_1E_1),
    \label{5.2.4}  \end{equation}
which replaces the ordinary cubical relations for degeneracies.

	In general, an $n$-cospan $ u = (X(\t), u(j, \t)) $ (with the notation of \ref{3.1}) has an $i$-directed 
{\it cylindrical} degeneracy (for $ i = 1,..., n), $ which is an $(n+1)$-cospan:
%
%  cylindrical i-degeneracy
    \begin{equation}
    \begin{array}{ll}
E_i(u)  =  (X(e_i \t), E_i(u)(j, \t)),	   &   (j = 1,..., n+1; \;  \t \in \We^{n+1}), 
\\[3pt]
E_i(u)(j, \t)  =  u(j^{\sharp i} , e_i \t)   &   (j \neq  i),
\\[3pt]
E_i(u)(i, \t)  =  d^\al X(e_i \t) \c  X(e_i \t) \to IX(e_i \t)     \;  & (\al  = t_i \neq  0),
    \label{5.2.5}\end{array}
    \end{equation}
where $ e_i \c \We^{n+1} \to \We^n $ omits the $i$-th coordinate (see (\ref{1.4.5}))  and we let:
    \begin{equation}
j^{\sharp i}   =  j  \;\; (for \; j \le i), \qq   j^{\sharp i}   =  j - 1  \;\; (for \;  j > i).
    \label{5.2.6}  \end{equation}

	Finally, we have an invertible special {\it symmetry comparison for cylindrical degeneracies}. For 
every $n$-cube $ u, $ we have an invertible special $(n+2)$-map $ \si _1(u), $ which is natural on 
$n$-maps and has the following faces
    \begin{equation}
    \begin{array}{ll}
\si _1u \c E_1E_1(u) \to E_2E_1(u)   &  \text{({\it symmetry 1-comparison}),}
\\[3pt]
\dda_1 \si_1(u)  =  \dda_2 \si_1(u)  =  \id (E_1u),   \q  &
\dda_{j+2} \si_1(u)  =  \si_1(\dda_j u).
    \label{5.2.7}\end{array}
    \end{equation}

	Via transpositions, $\si_1$ generates all the other symmetry comparisons $ E_j.E_i \to E_{i+1}.E_j $   $ (j \le i)$.

\subsection{Weak equivalences and homotopy invariance}\label{5.3}
Extending a previous definition on ordinary 
topological cospans (II.2.8), we say that a transversal $n$-map $ f \c u \to v $ between cubical 
topological cospans is a {\it weak equivalence} if it is special (\ref{1.6}) and all its components are 
homotopy equivalences (in $ \Top). $

	Thus, for every cubical cospan $ u = (X(\t), u(j, \t)), $ we have an obvious weak equivalence
    \begin{equation}
	p_i(u) \c E_i(u) \to u,
    \label{5.3.1}  \end{equation}
consisting of  the special transversal map whose non-identity components are the projections $ IX(e_i \t) \to X(e_i \t)$.  In cubical degree 1, 
$ p_1 \c E_1(X) \to e_1(X) $ has components $ (1, eX, 1)$
%
%  diagram for  p_1
    \begin{equation}
    \begin{array}{c}  \xymatrix  @C=27pt @R=25pt
{
~X~   \ar[r]^-{d^-}    \ar@{=}[d]  &  ~IX~   \ar[d]^e  &
~X~   \ar[l]_-{d^+}   \ar@{=}[d]
\\
~X~   \ar[r]_-{\id}     &  ~X~  &  ~X~   \ar[l]^-{\id}      
}
    \label{5.3.2} \end{array} \end{equation}

	We have already seen that such maps can not be considered as `homotopy equivalences' in $ 
\CCosp(\Top), $ because there are no transversal maps backwards (II.1.6).

	We say that two $n$-cubical topological cospans $ u, v $ are {\it weakly equivalent} if there exists 
a finite sequence of weak equivalences connecting them: $ u \to u_1 \lw u_2 \to ... \to u_n = v. $ Then, 
they must have the same vertices: $ u(\al _1,..., \al _n) = v(\al _1,..., \al _n), $ for $ \al _i = \pm .$

	A weak double functor $ F \c \CCosp_*(\Top) \to \AA, $ with values in an arbitrary weak cubical 
category, will be said to be {\it homotopy invariant} if: 

\Ndt   (i) it sends weak equivalences $ f \c u \to v $ between topological cospans to invertible (special) cells 
of $ \AA, $ and therefore weakly equivalent $n$-cubical cospans to isomorphic $n$-cubes of $ \A.$

\subsection{Review of homotopy pushouts}\label{5.4}
	The new concatenations will use a fundamental notion of 
homotopy theory, introduced by Mather \cite{Ma} (and also used in Part II).

	Let $ f \c A \to X$, $ g \c A \to Y $ form a span in $ \Top. $ The {\it standard homotopy pushout} 
{\it from} $ f $ {\it to} $ g $ is a four-tuple $ (P; u, v; \la) $ as in the left diagram below, where $ \la \c 
uf \to vg \c A \to P $ is a homotopy satisfying the following universal property (as for {\it cocomma 
squares} of categories), which determines the solution {\it up to homeomorphism}
%
% homotopy pushout
    \begin{equation}
    \begin{array}{l}  \xymatrix  @C=10pt @R=7pt
{
&&& \qq &&&  A    \ar[rr]^g   \ar[dd]_\dpp   &&   Y  \ar[dddd]^v
\\
A    \ar[rr]^g   \ar[dd]_f   &&   Y  \ar[dd]^v
\\
&&&&  A    \ar[rr]^\dm   \ar[dd]_f   &&   IA  \ar[rrdd]^\la
\\
X \ar[rr]_u &   \ar@/^/[ru]^\la   &	P	 &&&&&&
\\
&&&&  X   \ar[rrrr]_u   &&&&	P	
}
    \label{5.4.1} \end{array} \end{equation}

\Ndt   (a)  for every homotopy $ \la' \c u'f \to v'g \c A \to W, $ there is precisely one map $ h \c P \to W $ such that $ u' 
= hu$, $ v' = hv$, $ \la' = h\la$.

	(Writing $ h\la $ we are using the obvious {\it whisker composition} of homotopies and maps.) In 
$ \Top, $ the solution always exists and can be constructed as the ordinary colimit of the right-hand 
diagram above. This construction is based on the cylinder $ IA = A \ti  [0, 1] $ and its faces
    \begin{equation}
\dm, \dpp \c A \to IA,	\q	\dm(a)  =  (a, 0),    \;\;\;  \dpp(a)  =  (a, 1)   \qq  (a \in A).
    \label{5.4.2}  \end{equation}

	Therefore, the space $ P $ is a pasting of $ X $ and $ Y $ with the cylinder $ IA, $ and can be 
realised as a quotient of their topological sum, under the equivalence relation which gives the following 
identifications:
    \begin{equation}
P  =  (X + IA + Y)/\sim,	\qq   [f(a)]  =  [a, 0], \;\;\; [g(a)]  =  [a, 1]	  \qq   (a \in A).
    \label{5.4.3}  \end{equation}

	The term `standard homotopy pushout' will generally refer to this particular construction. Notice 
that, if $ f $ and $ g $ are topological embeddings, the spaces $ 
X $ and $ Y $ are embedded in $ P$.

	As a crucial feature, this construction always has strong properties of homotopy invariance (e.g., 
see \cite{G1}, Section 3), which an ordinary pushout need not have. Notice also that the cylinder $ IA 
$ is itself the standard homotopy pushout from $ \id A $ to $ \id A.$

\subsection{Cylindrical concatenation}\label{5.5}
	By definition, the {\it cylindrical i-concatenation} $ u \te _i v $ of 
$i$-consecutive $n$-cospans is computed on the $i$-concatenation model $ \We^{ni}_2 $ (of 
(\ref{1.4.6})), as 
    \begin{equation}
u \te_i v  =  [[u, v]].k_i \c \We^n \to \We^{ni}_2  \to \Top   \q  (\ddp_i (u) = \ddm_i (v)),
    \label{5.5.1}  \end{equation}
where $ [[u, v]] \c \We^{ni}_2 \to \Top $ sends all distinguished pushouts into standard homotopy 
pushouts (and, obviously, restricts to $u$ and $v$ on $ k^\al_i \c \We^n \to \We^{ni}$).

	(One could further formalise this by introducing the category $ {\rm hpt}\Cat $ of {\it h-categories} with 
distinguished homotopy pushouts, where an h-category - or a {\it category with homotopies} - is a 
category enriched on reflexive graphs with a suitable monoidal structure, as defined in \cite{G1}.)

\subsection{Comparisons for identities, associativity and interchange}\label{5.6}
The new structure $ 
\CCOSP_*(\Top) $ comes with various comparison maps, which make it a {\it symmetric quasi 
cubical category}, according to a definition which can be found in the last section.

\Ndt   (a) First, there are {\it lax} comparisons for identities, which are generally not invertible (but weak 
equivalences)
    \begin{equation}
\la_i u \c  E_i(\ddm_i u) \te _i u  \to u,	  \qq    \rho_iu \c  u \te _i E_i(\ddp_i u)  \to u.
    \label{5.6.1}  \end{equation}
	They are defined in the obvious way: for $ \la_i u, $ one collapses to its basis the {\it two} cylinders 
on $ \ddm_i u $ which we have pasted with $ u $ (i.e., the one appearing in $ E_i(\ddm_i u) $ and the 
one produced by the cylindrical concatenation $ \te _i$).

	Notice that $ E_i(X) \te _i E_i(X) \iso E_i(X), $ {\it but} $ \la E_i(X) $ and $ \rho E_i(X) $ are 
different and not invertible: they collapse different parts of the resulting cylinder on $ X. $ On the 
other hand, ordinary degeneracies work even worse (with homotopy pushouts): $ e_i(X) \te _i e_i(X) \iso 
E_i(X), $ which is only weakly equivalent to $ e_i(X).$

(b) Second, the cubical relation for pure degeneracies (which, in the presence of transpositions, can be 
reduced to the identity $ e_1.e_1 = e_2.e_1) $ does not hold. It is replaced with an invertible {\it 
symmetry comparison}, defined in \ref{5.2}
    \begin{equation}
    \begin{array}{ll}
\si _1u \c E_1E_1(u) \to E_2E_1(u),
\\[3pt]
\dda_1 \si _1(u)  =  \dda_2\si _1(u)  =  \id (E_1u), \q &
\dda_{j+2} \; \si _1(u)  =  \si _1(\dda_j u),
    \label{5.6.2}\end{array}
    \end{equation}
which, via transpositions, generates all the other ones, $ E_j.E_i \to E_{i+1}.E_j $$ (j \le i)$.

\Ndt   (c) Associativity of cylindrical concatenations works up to isomorphism (as in the weak cubical case):
    \begin{equation}
\ka_i(u, v, w) \c u \te _i (v \te _i w) \to (u \te _i v) \te _i w.
    \label{5.6.3}  \end{equation}

	This is expressed by the following computation (for ordinary cospans):
    \begin{equation}
	[\Xz + IA + [\Yz + IB + Z0]]  \iso  [\Xz + IA + \Yz + IB + Z^0]  \iso  [[\Xz + IA + \Yz] + IB + Z^0],
    \label{5.6.4}  \end{equation}
where the brackets $ [...] $ stand for a quotient modulo the adequate equivalence relations.

\Ndt   (d) Middle-four interchange also works up to isomorphism
    \begin{equation}
\chi _1(x, y, z, u) \c  (x \te _1 y) \te _2 (z \te _1 u)  \to  (x \te _2 z) \te _1 (y \te _2 u) \c \We^2 \to \Top.
    \label{5.6.5}  \end{equation}
	Indeed, the following diagram of pastings shows that both quaternary operations above are isomorphic to a 
symmetric one, denoted as $ \te _{12}(x, y, z, u)$

\bigskip

    \xymatrix  @C=15pt @R=20pt
{
\Bu   \ar@{-}[r]    \ar@{-}[d]   \ar@{}[rd]|x  &
\Bu   \ar@{-}[r]    \ar@{-}[d]   \ar@{}[rd]|{E_1a}  &
\Bu   \ar@{-}[r]    \ar@{-}[d]   \ar@{}[rd]|y  &
\Bu    \ar@{-}[d]  &
\Bu   \ar@{-}[r]    \ar@{-}[d]   \ar@{}[rd]|x  &
\Bu   \ar@{-}[r]    \ar@{-}[d]   \ar@{}[rd]  &
\Bu   \ar@{-}[r]    \ar@{-}[d]   \ar@{}[rd]|y  &
\Bu    \ar@{-}[d]  &
\Bu   \ar@{-}[r]    \ar@{-}[d]   \ar@{}[rd]|x  &
\Bu   \ar@{-}[r]    \ar@{-}[d]   \ar@{}[rd]|{E_1a}  &
\Bu   \ar@{-}[r]    \ar@{-}[d]   \ar@{}[rd]|y  &
\Bu    \ar@{-}[d]  &
\\ 
% line 2
\Bu   \ar@{-}[r]   \ar@{-}[d]    &
\Bu   \ar@{-}[r]   \ar@{}[rd]|{E_2h}  &
\Bu   \ar@{-}[r]    &
\Bu    \ar@{-}[d]  &
\Bu   \ar@{-}[r]    \ar@{-}[d]   \ar@{}[rd]|{E_2c}  &
\Bu                       \ar@{-}[d]   \ar@{}[rd]|{E_2h'}  &
\Bu   \ar@{-}[r]    \ar@{-}[d]   \ar@{}[rd]|{E_2d}  &
\Bu    \ar@{-}[d]  &
\Bu   \ar@{-}[r]    \ar@{-}[d]   \ar@{}[rd]|{E_2c}  &
\Bu   \ar@{-}[r]    \ar@{-}[d]   \ar@{}[rd]|w  &
\Bu   \ar@{-}[r]    \ar@{-}[d]   \ar@{}[rd]|{E_2d}  &
\Bu    \ar@{-}[d]  &
\\ 
% line 3 
\Bu   \ar@{-}[r]    \ar@{-}[d]   \ar@{}[rd]|z  &
\Bu   \ar@{-}[r]    \ar@{-}[d]   \ar@{}[rd]|{E_1b}  &
\Bu   \ar@{-}[r]    \ar@{-}[d]   \ar@{}[rd]|u  &
\Bu    \ar@{-}[d]  &
\Bu   \ar@{-}[r]    \ar@{-}[d]   \ar@{}[rd]|z  &
\Bu                       \ar@{-}[d]   \ar@{}[rd]  &
\Bu   \ar@{-}[r]    \ar@{-}[d]   \ar@{}[rd]|u  &
\Bu    \ar@{-}[d]  &
\Bu   \ar@{-}[r]    \ar@{-}[d]   \ar@{}[rd]|z  &
\Bu   \ar@{-}[r]    \ar@{-}[d]   \ar@{}[rd]|{E_1b}  &
\Bu   \ar@{-}[r]    \ar@{-}[d]   \ar@{}[rd]|u  &
\Bu    \ar@{-}[d]  &
\\ 
% line 4
\Bu   \ar@{-}[r]   &   \Bu   \ar@{-}[r]   &   \Bu   \ar@{-}[r]   &   \Bu   &
\Bu   \ar@{-}[r]   &   \Bu   \ar@{-}[r]   &   \Bu   \ar@{-}[r]   &   \Bu   &
\Bu   \ar@{-}[r]   &   \Bu   \ar@{-}[r]   &   \Bu   \ar@{-}[r]   &   \Bu   &
}
    \begin{equation}
\;\:\;\:  (x \te _1 y) \te _2 (z \te _1 u)  \qq\;\:
(x \te _2 z) \te _1 (y \te _2 u)  \qq  \te _{12}(x, y, z, u). 
    \label{5.6.6}  \end{equation}

	Above, we have written: 
    \begin{equation}
    \begin{array}{ccc}
a  =  \ddp_1x  =  \ddm_1y,	   &   b  =  \ddp_1z  =  \ddm_1u,   &   h  =  \ddp_2x \te _1 \ddp_2y,
\\[3pt]
c  =  \ddp_2x  =  \ddm_2z,	  &   d  =  \ddp_2y  =  \ddm_2u,	  &   h'  =  \ddp_1x \te _1 \ddp_1z,
\\[3pt]
&  w  =  E_1E_1(v)  \iso  E_2E_1(v),
    \label{5.6.7}\end{array}
    \end{equation}
where $ v = \ddp_1\ddp_2x $ is the $(n-2)$-cospan common to the four given items $ x, y, z, u.$

	The symmetric property of the operation $ \te _{12} $ is:
    \begin{equation}
\te _{12}(x, y, z, u).s_1  =  \te _{12}(xs_1, zs_1, ys_1, us_1).
    \label{5.6.8}  \end{equation}

\Ndt  (e) Finally, we have an invertible {\it nullary interchange comparison}, for 1-consecutive $n$-cubes $ x, y $ 
    \begin{equation}
\io_1(x, y) \c  E_1(x) \te _2 E_1(y)  \to  E_1(x \te _1 y).
    \label{5.6.9}  \end{equation}

	It can be constructed using the isomorphic construction of $ E_1(x \te _1 y) $ displayed below, and 
the symmetry isomorphism (\ref{5.6.2}) (again, we write $ a = \ddp_1x = \ddm_1y)$
    \begin{equation}
    \begin{array}{c}      \xymatrix  @C=1pt @R=20pt
{
\Bu   \ar@{-}[rr]    \ar@{-}[d]_x   \ar@{}[rrd]|{E_1x}  &&
\Bu    \ar@{-}[d]^x  &\qq&
\Bu   \ar@{-}[rr]    \ar@{-}[d]_x   \ar@{}[rrd]|{E_1x}  &&
\Bu    \ar@{-}[d]^x  &&
\\ 
% line 2
\Bu   \ar@{-}[rr]   \ar@{-}[d]_{E_1a}     \ar@{}[rrd]|{E_2E_1a} &&
\Bu    \ar@{-}[d]^{E_1a}   &&
\Bu   \ar@{-}[rr]    \ar@{-}[d]_{E_1a}    \ar@{}[rrd]|{E_1E_1a}  &&
\Bu    \ar@{-}[d]^{E_1a}  &\qq&
\Bu   \ar[rrr]^1    \ar[d]^2  &&&
\\ 
% line 3 
\Bu   \ar@{-}[rr]    \ar@{-}[d]_y   \ar@{}[rrd]|{E_1x}  &&
\Bu    \ar@{-}[d]^y  &&
\Bu   \ar@{-}[rr]    \ar@{-}[d]_y   \ar@{}[rrd]|{E_1y}  &&
\Bu    \ar@{-}[d]^y  && &&
\\ 
% line 4
\Bu   \ar@{-}[rr]   &&   \Bu   &&
\Bu   \ar@{-}[rr]   &&   \Bu   &&
\\
&  E_1(x) \te _2 E_1(y)   &&&&   E_1(x \te _1 y)  
}
    \label{5.6.10} \end{array} \end{equation}
\smallskip

\subsection{Cylindrical collared degeneracies}\label{5.7}
	We end this section by remarking that a sort of `cylindrical 
degeneracies' also exist for collared cospans, in $ \CCc_*(\Top).$

	Beginning at cubical degree 0, every space $ X $ has a {\it cylindrical collared degeneracy} (not to 
be confused with the cylindrical degeneracy $ E_1(X) $ considered above)
    \begin{equation}
    \begin{array}{l}
\E_1(X)  =  (X, IX, X;  \;\;  E^- \c IX \to IX \lw IX \cc E^+),
\\[3pt]
E^-(x, t)  =  (x, t/3),   	\qq   E^+(x)  =  (x, 1-t/3).
    \label{5.7.1}\end{array}
    \end{equation}

	Then, every collared cospan $ U = (\Xm, \Xz, \Xp; \;\;   \Um \c I\Xm \to \Xz \lw I\Xp \cc \Up) $ has two 
cylindrical collared degeneracies $ \E_1(U)$, $ \E_2(U) = \E_1(U).s, $ determined by the following 
collars 

\bigskip
 \xymatrix  @C=15pt @R=24pt
{
~I^2X^-~   \ar[r]^-{E^-IX^-}    \ar[d]_{IU^-}  &  ~IX^-~   \ar[d]^{IU^-}  &
~I^2X^-~   \ar[l]_-{E^+IX^-}  \ar[d]^{IU^-}  
&&
~I^2X^-~   \ar[r]^-{IU^-}    \ar[d]_{E^-IX^-}  &  ~IX^0~   \ar[d]^{E^-X^0}  &
~I^2X^+~   \ar[l]_-{~IU^+}  \ar[d]^{E^-IX^+}  
\\
~IX^0~   \ar[r]^-{E^-X^0}     &  ~IX^0~  &  ~IX^0~   \ar[l]_-{E^+X^0}   
&&
~IX^-~   \ar[r]^-{IU^-}     &  ~IX^0~  &  ~IX^+~   \ar[l]_-{IU^+}   
&   \Bu    \ar[r]^1   \ar[d]^2  &
\\  
~I^2X^+~   \ar[r]_-{E^-IX^+}    \ar[u]^{IU^+}  &  ~IX^+~   \ar[u]_{IU^+}   &
~I^2X^+~   \ar[l]^-{E^+IX^+}   \ar[u]_{~IU^+} 
&&
~I^2X^-~   \ar[r]_-{IU^-}    \ar[u]^{E^+IX^-}   &  ~IX^0~   \ar[u]_{E^+X^0}    &
~I^2X^+~   \ar[l]^-{~IU^+}  \ar[u]_{E^+IX^+}  &
}
\smallskip
    \begin{equation}
  \E_1(U)   \qqq\qq\;\;   \E_2(U) =  \E_1(U).s.  \;\;
    \label{5.7.2}  \end{equation}

	Within manifolds and cobordism, $ \E_1(X) $ is generally used as the degenerate cobordism on the 
manifold $ X $ (cf. \cite{Mo}). But again the cubical relation $ e_1e_1 = e_2e_1 $ is {\it not} satisfied: 
$ \E_1\E_1(X) $ and $ \E_2\E_1(X) $ are different (and isomorphic):

\bigskip
 \xymatrix  @C=17pt @R=24pt
{
~I^2X~   \ar[r]^-{E^-I}    \ar[d]_{IE^-}  &  ~I^2X~   \ar[d]^{IE^-}  &
~I^2X~   \ar[l]_-{E^+I}  \ar[d]^{IE^-}  
&&
~I^2X~   \ar[r]^-{IE^-}    \ar[d]_{E^-I}  &  ~I^2X~   \ar[d]^{E^-I}  &
~I^2X~   \ar[l]_-{~IE^+}  \ar[d]^{E^-I}  
\\
~I^2X~   \ar[r]^-{E^-I}     &  ~I^2X~  &  ~I^2X~   \ar[l]_-{E^+I}   
&&
~I^2X~   \ar[r]^-{IE^-}     &  ~I^2X~  &  ~I^2X~   \ar[l]_-{IE^+}   
&   \Bu    \ar[r]^1   \ar[d]^2  &
\\  
~I^2X~   \ar[r]_-{E^-I}    \ar[u]^{IE^+}  &  ~I^2X~   \ar[u]_{IE^+}   &
~I^2X~   \ar[l]^-{E^+I}   \ar[u]_{~IE^+} 
&&
~I^2X~   \ar[r]_-{IE^-}    \ar[u]^{E^+I}   &  ~I^2X~   \ar[u]_{E^+I}    &
~I^2X~   \ar[l]^-{~IE^+}  \ar[u]_{E^+I}  &
}
\smallskip
    \begin{equation}
  \E_1\E_1(X)   \qqq\qq   \E_2\E_1(X).  \qq
    \label{5.7.3}  \end{equation}

\section{Cobordisms}\label{6}

	We obtain here a {\it quasi} cubical category of $k$-manifolds and cubical cobordisms, based on 
the notion of a differentiable manifold with faces \cite{Do, Ja, La}.

\subsection{Goals and problems}\label{6.1}
	We construct now the {\it quasi} cubical subcategory $ 
\CCOB_*(k) \subset \CCOSP_*(\Top) $ of $k$-manifolds and cubical cobordisms, with cylindrical 
degeneracies and concatenation.

	The 2-cubical truncation $ 2\CCOB_*(k) $ of our construction is related with the construction of Morton and Baez  \cite{Mo, Ba}, which works with assigned collars, cylindrical degeneracies and concatenation by pushout. But notice that, here, $k$ denotes the topological dimension of the objects, while papers dealing with a 2-cubical or 2-globular structure generally refer to the dimension $k+2$ of the highest cobordisms which appear in the structure itself. In the unbounded cubical (or globular) case, there is no upper bound for such dimensions. 

	{\it Likely}, working with assigned collars, one can also construct a {\it weak} cubical category $ 
\CCob_*(k) \subset  \CCc_*(\Top). $ However, the technical aspects of this construction seem to be 
so heavy, that one wonders whether such complication would be justified by the advantage of obtaining 
a less weak structure (satisfying all cubical axioms and having all comparisons invertible).

\subsection{Manifolds with corners}\label{6.2}
We begin by recalling some basic definitions. A {\it differentiable 
manifold with corners} \cite{Do, La} is a second-countable Hausdorff space $ X $ which admits a 
differentiable atlas of charts 
%
% charts
    \begin{equation}
	\ph _i \c U_i \to \R^n_+ =  [0, \infty[^n.
    \label{6.2.1}  \end{equation}
	(These charts are homeomorphisms from open subspaces of $ X $ onto open subspaces of the 
euclidean {\it sector} $ \R^n_+$, and the changes of charts $ \ph_i \ph_j^{-1} $ are diffeomorphisms, 
in the obvious sense.)

	Every point $ x \in X $ has a well-defined {\it index} $ c(x) $ between  0  and $ n, $ which counts 
the number of null coordinates of $ \ph_i (x), $ for whichever chart $ \ph_i  $ defined at $ x. $ Thus, $ c(x) 
= 0 $ means that $ X $ is locally euclidean at the point $ x. $ By definition, a {\it connected face} of $ 
X $ is the closure of a connected component of the subset of points of index 1; $ \dd X $ is the union of 
all connected faces. 

	A {\it manifold with faces} \cite{Ja, La} is a manifold with corners where every point $ x $ belongs 
to precisely $ c(x) $ connected faces. Every manifold with boundary is a manifold with faces, where 
the highest possible index is 1. A compact cube is also a manifold with faces: it has six connected 
faces, vertices have index 3, the other edge-points have index 2 and the remaining face-points index 1. 
A {\it face} of a manifold with faces is a union of {\it disjoint} connected faces, and is still a manifold 
with faces; for instance, a compact cube has three non-connected faces.

	Finally, a {\it manifold with n (distinguished) faces} $ X = (X; \dd _1X,..., \dd _nX) $ (\cite{Ja, La}, 
where it is called an {\it $\langle n \rangle$-manifold}) is a manifold with faces equipped with an indexed family of $ n $ faces 
which cover $ \dd X $ and such that $ \dd _iX \cap \dd _jX $ is always a face of $ \dd _i X $ (for $ i \neq  j). $ 
Plainly, $ n $ is at most equal to the number of connected faces of $ X; $ if it is less, the structure we 
are considering is not determined by $ X.$

	A {\it morphism} $ f \c X \to Y $ of such manifolds, with values in $ Y = (Y; \dd _1Y,..., \dd _mY), $ will 
be a {\it continuous} mapping which sends each face of $ X $ into some face of $ Y. $ (The papers 
referred to above consider {\it differentiable} maps; the present choice will simplify the relations with 
topological cospans.)

	The categorical sum is:
    \begin{equation}
	X + Y  =  (X + Y;  \dd _1X,..., \dd _nX, \dd _1Y,..., \dd _mY).
    \label{6.2.2}  \end{equation}

\subsection{Cubical cospans of manifolds}\label{6.3}
	On this basis, it is easy to define the transversally full quasi 
cubical subcategory $ \CCOB_*(k) \subset \CCOSP_*(\Top) $ of $k$-manifolds.

	A {\it cubical cospan} of $k$-manifolds $ u \c \We^n \to \Top $ is a topological cospan
$$
u =  (X(\t), u(i, \t)),
$$
 determined by a manifold with faces $ X = (X; (\dda_i X)), $ as made explicit below; $ 
X $ has dimension $ k+n $ and $ 2n $ distinguished faces, which are pairwise disjoint: $ \ddm_i X \cap 
\ddp_i X = \es. $

	Namely, the central space $ X(0,...0) $ of $ u $ is $ X $ itself, all the other spaces are intersections 
of the assigned faces (and faces as well) and all maps of $ u $ are inclusions
    \begin{equation}
    \begin{array}{ll}
X(\t)  =  \dd^{t_1}_1 X \cap .... \cap \dd^{t_n}_nX  \q   &   (\t = (t_1,..., t_n) \in \We^n),
\\[3pt]
u(i, \t) \c X(\t)  \subset X(\t^{\sharp i} )	   &   (t_i \neq  0),
    \label{6.3.1}\end{array}
    \end{equation}
where we let $ \dd^0_i X = X, $ for all $ i.$
	Thus, the central cospan in each direction $ i $ is given by the inclusion in $ X $ of its two $i$-faces
    \begin{equation}
\ddm_iX \to X \lw \ddp_i X,
    \label{6.3.2}  \end{equation}

	Notice that such a particular {\it topological cospan} $ u \c \We^n \to \Top $ fully determines $ X 
$ and the family $ (\dda_i X) $ of its faces. Moreover, all the squares in $ u $ along two arbitrary 
directions are pullbacks.

	By definition, a {\it transversal map} $ f \c u \to v $ of such cubical cospans is an arbitrary natural 
transformation $ f \c u \to v \c \We^n \to \Top; $ as a consequence, it has an underlying morphism $ f 
\c X \to Y $ of manifolds with distinguished faces (a continuous mapping).

	Plainly $ \CCOB_*(k) $ is closed in $ \CCOSP_*(\Top) $ under faces, (cylindrical) degeneracies, 
transpositions and concatenations. Being transversally full, it automatically contains the comparisons 
for identities, associativity and (cubical) interchanges; recall that the comparisons of identities are {\it lax}.

	Since manifolds with faces have collars (\cite{La}, Lemma 2.1.6), we always have 
$ E_i(u) \te _i u \iso  u. $ But this isomorphism depends on the choice of collars for $ u, $ and is not natural (in the present 
structure). On the other hand, the natural comparison inherited from $ \CCOSP_*(\Top) $
 `collapses cylinders' and is not invertible.

\section{Symmetric quasi cubical categories}\label{7}

	We make precise the definition of a symmetric {\it quasi} cubical category, extending the notion of a 
symmetric {\it weak} cubical category. (The latter, introduced in Part I, is recalled here in Section 1).

\subsection{Symmetric quasi pre-cubical categories}\label{7.1}
A symmetric {\it quasi} pre-cubical category
    \begin{equation}
\AA  =  ((A_n), (\dda_i ), (e_i), (s_i), (+_i)),
    \label{7.1.1}  \end{equation}
is a symmetric {\it quasi} cubical set (\ref{1.1}, \ref{1.2}), equipped with the following additional 
operations.

	For $ 1 \le i \le n, $ the $i$-{\it concatenation} $ x +_i y $ (or $i$-composition) of two $n$-
cubes $ x, y $ is defined when $ x, y $ are {\it $i$-consecutive}, i.e. $ \ddp_i (x) = \ddm_i(y), $ and satisfies 
the following `geometrical' interactions with faces and transpositions
    \begin{equation}
    \begin{array}{ll}
\ddm_i(x +_i y)  =  \ddm_i(x),   &   \ddp_i (x +_i y)  =  \ddp_i (y),
\\[3pt]
\dda_j (x +_i y)  =  \dda_j (x) +_{i-1} \dda_j (y)	& (j < i),
\\[3pt]
 \qq\;\;                =  \dda_j (x) +_i \dda_j (y)   &   (j > i),
    \label{7.1.2}\end{array}
    \end{equation}
    \begin{equation}
    \begin{array}{ll}
s_{i-1}(x +_i y)  =  s_{i-1}(x) +_{i-1} s_{i-1}(y),	\q   &
s_i(x +_i y)  =  s_i(x) +_{i+1} s_i(y),
\\[3pt]
s_j(x +_i y)  =  s_j(x) +_i s_j(y)	   &   (j \neq  i-1, i).
    \label{7.1.3}\end{array}
    \end{equation}

	Again, we are not (yet) assuming categorical or interchange laws for the $i$-compositions. Our 
structure is a symmetric pre-cubical category (as defined in I.3.4) if its degeneracies satisfy the cubical 
relations (\ref{1.1.3}) (or (\ref{1.2.3})) and agree with concatenations:
    \begin{equation}
    \begin{array}{ll}
e_j(x +_i y)  =  e_j(x) +_{i+1} e_j(y)	  &   (j \le i \le n),
\\[3pt]
 \qq\;           =  e_j(x) +_i e_j(y)	 &   (i < j \le n+1).
    \label{7.1.4}\end{array}
    \end{equation}

	The presence of transpositions allows us to reduce condition (\ref{7.1.4}) to:
    \begin{equation}
e_1(x +_1 y)  =  e_1(x) +_2 e_1(y).
    \label{7.1.5}  \end{equation}

\subsection{Introducing transversal maps. }\label{7.2}
As in I.4.1, we introduce now a richer structure, having $n$-dimensional 
maps {\it in a new direction} 0, which can be viewed as strict or `transversal' in opposition 
with the previous weak or `cubical' directions. The comparisons for units, associativity and interchange 
will be maps of this kind.

	Let us start with considering a general {\it category object} $ \AA $ within the category of 
symmetric quasi pre-cubical categories and their functors
    \begin{equation}
    \begin{array}{c}  \xymatrix  @C=30pt @R=20pt
{
~\AA^0~   \ar@<-.8ex>[r]_{e_0}      &
~\AA^1~   \ar[l]  \ar@<-.8ex>[l]_{\dda_0}    &
~\AA^2~   \ar[l]_{c_0}
}
    \label{7.2.1} \end{array} \end{equation}
\smallskip

	We have thus:
	
\Ndt (qcub.1) A symmetric quasi pre-cubical category $ \AA^0 = ((A_n), (\dda_i ), (e_i), (s_i), (+_i)), $ 
whose entries are called {\it n-cubes}, or {\it n-dimensional objects} of $ \AA.$
	
\Ndt (qcub.2) A symmetric quasi pre-cubical category $ \AA^1 = ((M_n), (\dda_i ), (e_i), (s_i), (+_i)), $ 
whose entries are called {\it n-maps} or {\it n-dimensional maps} of $ \AA.$
	
\Ndt (qcub.3) Symmetric cubical functors $ \dda_0 $ and $ e_0, $ called 0-{\it faces} and 0-{\it degeneracy}, 
with $ \dda_0.e_0 = \id .$

	Typically, an $n$-map will be written as $ f \c x \to x', $ where $  \ddm_0 f = x, $
$ \ddp_0 f = x' $ are $n$-cubes. Every $n$-dimensional object $ x $ has an {\it identity} $ e_0(x) \c x \to x. $ Note that $ \dda_0 
$ and $ e_0 $ preserve cubical faces $ (\dda_i , $ with $ i > 0), $ cubical degeneracies $ (e_i), $ 
transpositions $ (s_i) $ and cubical concatenations $ (+_i). $ In particular, given two $i$-consecutive 
$n$-maps $ f, g, $ their 0-faces are also $i$-consecutive and we have:
    \begin{equation}
	f +_i g \c  x +_i y \to  x' +_i y'	\q	(f \c x \to x',  \; g \c y \to y'; \;  \ddp_i f = \ddm_ig).
    \label{7.2.2}  \end{equation}

\Ndt   (qcub.4) A composition law $ c_0 $ which assigns to two 0-consecutive $n$-maps $ f \c x \to x', $
$ h \c x' \to x'' $ an $n$-map $ hf \c x \to x'' $ (also written $ h.f). $ This composition law is (strictly) 
categorical, and forms a category $ \A_n = (A_n, M_n, \dda_0, e_0, c_0). $ It is also consistent with the 
symmetric quasi pre-cubical structure, in the following sense
    \begin{equation}
    \begin{array}{ll}
\dda_i (hf)  =  (\dda_i h).(\dda_i f),   &    e_i(hf)  =  (e_ih)(e_if),
\\[3pt]
s_i(hf)  =  (s_ih)(s_if),   &   (h +_i k).(f +_i g)  =  hf +_i kg,
    \label{7.2.3}\end{array}
    \end{equation}
    
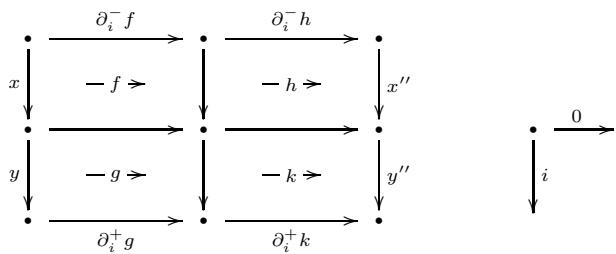
\xymatrix  @C=8pt @R=10pt
{
\qqq &  \Bu   \ar[rrrr]^{\ddm_i f}   \ar[dd]_x  &&&&  \Bu   \ar[rrrr]^{\ddm_i h}   \ar[dd]   &&&&  \Bu    \ar[dd]^{x''}
\\ 
&&    \ar[rr]|{\; f \;}   &&&&  \ar[rr]|{\; h \;}   &&
\\
& \Bu   \ar[rrrr]   \ar[dd]_y  &&&&  \Bu   \ar[rrrr]   \ar[dd]   &&&&  \Bu    \ar[dd]^{y''}
& \q &  \Bu \ar[rr]^0  \ar[dd]^i  &&
\\ 
&&    \ar[rr]|{\; g \;}   &&&&  \ar[rr]|{\; k \;}   &&
\\
& \Bu   \ar[rrrr]_{\ddp_i g}    &&&&   \Bu   \ar[rrrr]_{\ddp_i k}   &&&&  \Bu   &&&
}
\smallskip

	The last condition, represented in the diagram above, is the (strict) middle-four interchange between the strict composition $ c_0 $ and 
any weak one. An $n$-map $ f \c x \to x' $ is said to be {\it special} if its $ 2^n $ vertices are identities
    \begin{equation}
\dd^\bal f \c  \dd^\bal x \to \dd^\bal x'	\qq  
\dd^\bal   =   \dd^{\al_1}_1  \dd^{\al_2}_2 ...  \dd^{\al_n}_n   \q  (\al _i = \pm ).
    \label{7.2.4}  \end{equation}
%(4)	 
	In degree 0, this just means an identity.

\subsection{Comparisons}\label{7.3}
	Extending I.4.2, we can now define a {\it symmetric quasi cubical category} $ 
\AA $ as a category object within the category of symmetric quasi pre-cubical categories and 
symmetric cubical functors, which is further equipped with special transversal maps, playing the role of 
comparisons for units, symmetry, associativity and cubical interchange, as follows. (We only assign the 
comparisons in direction 1; all the others can be obtained with transpositions. Notice also that the unit 
comparisons are not assumed to be invertible.)

\Ndt   (qcub.5)	For every $n$-cube $ x, $ we have special $n$-maps $ \la_1x $ and $ \rho_1x, $ which are 
natural on $n$-maps and have the following faces (for $ n > 0)$
%
% left-unit 1-comparison
    \begin{equation}
    \begin{array}{llr}
\la_1x\c (e_1\ddm_1 x) +_1 x  \to  x   &&    \mbox{({\it left-unit 1-comparison}),}
\\[3pt]
\dda_1 \la_1x  =  e_0 \dda_1 x,	  &  \dda_j \la_1x  =  \la_1\dda_j x  \;\; &    (1 < j \le  n),
    \label{7.3.1}\end{array}
    \end{equation}

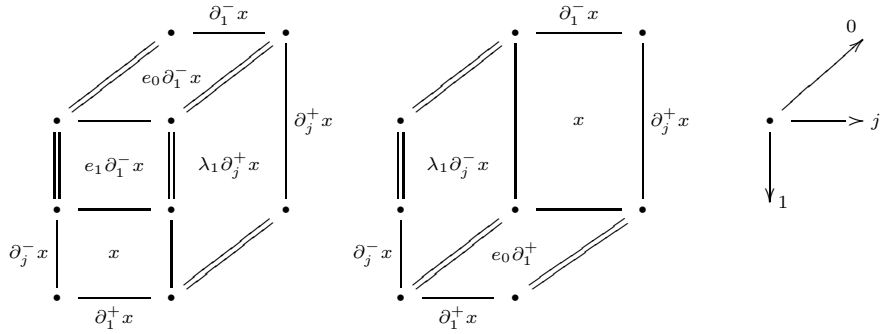
     \xymatrix  @C=27pt @R=25pt  
{  
&&   \Bu  \ar@{-}[r]^{\ddm_1 x}   \ar@{}[d]|{e_0\ddm_1 x}   &
\Bu     \ar@{-}[dd]^{\ddp_j x}     &&
\Bu  \ar@{-}[r]^{\ddm_1 x}   \ar@{-}[dd]  \ar@{}[rdd]|x  &
\Bu     \ar@{-}[dd]^{\ddp_j x}    &&
\\ 
&\Bu  \ar@{-}[r]   \ar@{=}[d]    \ar@{=}[ru]   \ar@{}[rd]|{e_1\ddm_1 x}  &
 \Bu  \ar@{=}[ru]  \ar@{=}[d]    \ar@{}[rd]|{ \la_1\ddp_j x }  &&
\Bu \ar@{=}[d]    \ar@{=}[ru]    \ar@{}[rd]|{ \la_1\ddm_j x }   &&  \q &
\Bu  \ar[r]|>{~~~j}   \ar[d]^>1   \ar[ru]^>0  &
\\
&\Bu  \ar@{-}[r]   \ar@{-}[d]_{\ddm_j x}   \ar@{}[rd]|x   &
\Bu   \ar@{-}[d] &   \Bu   &  
\Bu  \ar@{-}[d]_{\ddm_j x}   &
\Bu   \ar@{-}[r]    \ar@{}[d]|{e_0 \ddp_1}   &   \Bu  &&
\\  
& \Bu  \ar@{-}[r]_{\ddp_1 x}   & \Bu  \ar@{=}[ru]   &&
\Bu  \ar@{-}[r]_{\ddp_1 x}    \ar@{=}[ru]    & \Bu  \ar@{=}[ru] 
}
\smallskip

%
% right-unit 1-comparison
    \begin{equation}
    \begin{array}{llr}
\rho_1x\c  x +_1 (e_1\ddp_1 x)  \to  x,   &&    \mbox{({\it right-unit 1-comparison}),}
\\[3pt]
\dd^\al_1 \rho_1x  =  e_0  \dd^\al_1 x,	&   \dd^\al_j \rho_1x  =  \rho_1\dd^\al_jx   \;\; &    (1 < j \le  n),
    \label{7.3.2}\end{array}
    \end{equation}
\smallskip

     \xymatrix  @C=27pt @R=25pt  
{  
&&   \Bu  \ar@{-}[r]^{\ddm_1 x}    \ar@{}[d]|{e_0\ddm_1 x}  & 
\Bu     \ar@{-}[dd]^{\ddp_j x}  &&
\Bu  \ar@{-}[r]^{\ddm_1 x}   \ar@{-}[dd]  \ar@{}[rdd]|x  &
 \Bu     \ar@{-}[dd]^{\ddp_j x}  &&
\\  
&  \Bu  \ar@{-}[r]   \ar@{-}[d]_{\ddm_j x}    \ar@{=}[ru]  \ar@{}[rd]|x   &
\Bu  \ar@{=}[ru]  \ar@{-}[d]   \ar@{}[rd]|{\rho_1\ddp_j x}  &&
\Bu \ar@{-}[d]_{\ddm_j x}    \ar@{=}[ru]  \ar@{}[rd]|{\rho_1\ddm_j x}  &   
& \q & \Bu  \ar[r]|>{~~~j}   \ar[d]^>1   \ar[ru]^>0  &&
\\
&\Bu  \ar@{-}[r]   \ar@{=}[d]   \ar@{}[rd]|{e_1\ddp_1 x}   &
\Bu   \ar@{=}[d]   &  \Bu &  \Bu  \ar@{=}[d]   &
\Bu   \ar@{-}[r]  \ar@{}[d]|{e_0 \ddp_1}   &   \Bu  &
\\  
& \Bu  \ar@{-}[r]_{\ddp_1 x}   & \Bu  \ar@{=}[ru]  &&
\Bu  \ar@{-}[r]_{\ddp_1 x}    \ar@{=}[ru]    & \Bu  \ar@{=}[ru] 
}
\smallskip

	(Notice that the 0-direction of $ \rho_1x $ is reversed, with respect to I.4.2 - where $ \rho_1x $ is 
invertible and its direction is inessential.)
	
\Ndt (qcub.6)  For every $n$-cube $ x, $ we have an invertible special $(n+2)$-map $ \si _1x, $ which is 
natural on $n$-maps and has the following faces (for $ n > 0)$
%
% symmetry 1-comparison
    \begin{equation}
    \begin{array}{lr}
\si _1x \c e_1e_1(x) \to e_2e_1(x)    &    \mbox{({\it symmetry 1-comparison}),}
\\[3pt]
 \dda_1 \si _1(x)  =   \dda_2 \si_1 (x)  =  \id (e_1x),	\qq   &
\dda_{j+2}  \si _1(x)  =  \si _1(\dda_j x).
    \label{7.3.3}\end{array}
    \end{equation}

\Ndt    (qcub.7)	For three 1-consecutive $n$-cubes $ x, y, z, $ 
we have an {\it invertible} special $n$-map $ \ka_1(x, y, z), $ 
which is natural on $n$-maps and has the following faces
%
% associativity 1-comparison
    \begin{equation}
    \begin{array}{lr}
\ka_1(x, y, z) \c  x +_1 (y +_1 z) \to  (x +_1 y) +_1 z    \; &    \mbox{({\it associativity 1-comparison}),}
\\[3pt]
\ddm_1\ka_1(x, y, z)  =  e_0\ddm_1x,	&   \ddp_1\ka_1(x, y, z)  =  e_0\ddp_1z,
\\[3pt]
\dda_j \ka_1(x, y, z)  =  \ka_1(\dda_j x, \dda_j y, \dda_j z)   &   (1 < j \le n),
    \label{7.3.4}\end{array}
    \end{equation}
\smallskip

     \xymatrix  @C=27pt @R=20pt  
{  
&&   \Bu  \ar@{-}[r]^{\ddm_1 x}      \ar@{}[d]|{e_0\ddm_1 x}   & 
\Bu     \ar@{-}[d]^{\ddp_j x}   &&
\Bu  \ar@{-}[r]^{\ddm_1 x}   \ar@{-}[d]    \ar@{}[rdd]|{x +_1 y}  &
\Bu     \ar@{-}[d]^{\ddp_j x}  
\\       
& \Bu  \ar@{-}[r]   \ar@{-}[dd]_{\ddm_j x}    \ar@{=}[ru]     \ar@{}[rdd]|x  & 
\Bu  \ar@{=}[ru]   \ar@{-}[dd]   \ar@{}[rddd]|{ \ka_1\ddp_j}  &
\Bu \ar@{-}[d]^{\ddp_j y}   &
\Bu \ar@{-}[dd]    \ar@{=}[ru]   \ar@{}[rddd]|{ \ka_1\ddm_j}  &
\Bu    \ar@{-}[d]   & \Bu \ar@{-}[d]^{\ddp_j y}  &&&
\\
&&&   \Bu  \ar@{-}[dd]^{\ddp_j z}   &&
\Bu   \ar@{-}[r]  \ar@{-}[dd]     \ar@{}[rdd]|z   &   
\Bu   \ar@{-}[dd]^{\ddp_j z}
& \Bu  \ar[r]|>{~~~j}   \ar[d]^>1   \ar[ru]^>0  &
\\
& \Bu  \ar@{-}[r]   \ar@{-}[d]_{\ddm_j y}    \ar@{}[rdd]|{ y +_1 z}   & 
\Bu    \ar@{-}[d]   &&   \Bu \ar@{-}[d]   &&&&&
\\
& \Bu    \ar@{-}[d]_{\ddm_j z}  & \Bu \ar@{-}[d]  & \Bu  & 
\Bu   \ar@{-}[d]   &
\Bu   \ar@{-}[r]  \ar@{}[d]|{e_0 \ddp_1}   &   \Bu  
\\  
& \Bu  \ar@{-}[r]_{\ddp_1 x}   & \Bu  \ar@{=}[ru]   &&   
\Bu  \ar@{-}[r]_{\ddp_1 x}    \ar@{=}[ru]    & \Bu  \ar@{=}[ru] 
}
\smallskip

\Ndt (wcub.8)	Given four $n$-cubes $ x, y, z, u $ which make the concatenations below well-formed, we have 
an {\it invertible} special $n$-map $ \chi _1(x, y, z, u)$,  which is natural on $n$-maps and has the 
following faces (partially displayed below)

$$
\chi_1(x, y, z, u)\c (x +_1 y) +_2 (z +_1 u)  \to  (x +_2 z) +_1 (y +_2 u)
\;   \mbox{({\it interchange 1-comparison}),}
$$
%
% interchange 1-comparison
    \begin{equation}
    \begin{array}{c}
\ddm_1 \chi_1(x, y, z, u)  =  e_0(\ddm_1 x +_2 \ddm_1 z),  \;\;
\ddp_1 \chi_1(x, y, z, u)  =  e_0(\ddp_1 y +_2 \ddp_1 u),
\\[3pt]
\ddm_2 \chi_1(x, y, z, u)  =  e_0(\ddm_2 x +_1 \ddm_2 y),  \;\;
\ddm_2 \chi_1(x, y, z, u)  =  e_0(\ddm_2 x +_1 \ddm_2 y),
\\[3pt]
\dd^\al_j\chi_1(x, y, z, u)  =  \chi_1(\dd^\al_jx, \dd^\al_jy, \dd^\al_jz, \dd^\al_ju)  \qqq
(2 < j \le  n),
    \label{7.3.5}\end{array}
    \end{equation}

\smallskip
     \xymatrix  @C=11pt @R=20pt  
{  
%line 1
&&   \Bu  \ar@{-}[r]^{\ddm_2 x}   \ar@{}[d]|{e_0}    &
\Bu   \ar@{-}[r]^{\ddm_2 y}   &  \Bu   \ar@{-}[d]^{\ddp_1 y}  &&&&
\Bu  \ar@{-}[r]^{\ddm_2 x}   \ar@{-}[d]   \ar@{}[rd]|x  &
\Bu  \ar@{-}[r]^{\ddm_2 y}  \ar@{-}[d]   \ar@{}[rd]|y  &
\Bu   \ar@{-}[d]^{\ddp_1 y}  &&&
\\  
%line 2
\Bu  \ar@{-}[r]   \ar@{-}[d]_{\ddm_1 x}    \ar@{=}[rru]    \ar@{}[rrd]|{x +_1 y}  &
\Bu  \ar@{-}[r]  & \Bu  \ar@{=}[rru]  \ar@{-}[d]   \ar@{}[rrd]|{e_0}  &&  
\Bu \ar@{-}[d]^{\ddp_1 u}  &&
\Bu \ar@{-}[d]_{\ddm_1 x}    \ar@{=}[rru]    \ar@{}[rrd]|{e_0}    &&
\Bu \ar@{-}[d]     \ar@{}[r]|{+_2}    \ar@{}[rd]|z  &
\Bu \ar@{-}[d]       \ar@{}[r]|{+_2}    \ar@{}[rd]|u   & 
\Bu \ar@{-}[d]^{\ddp_1 u}  && 
\Bu   \ar[r]|>{~~~1}   \ar[d]^>2   \ar[ru]^>0  &
\\
%line 3
\Bu  \ar@{-}[r]   \ar@{-}[d]_{\ddm_1 z}     \ar@{}[rrd]|{z +_1 u}  &
\Bu  \ar@{-}[r]  & \Bu  \ar@{-}[d]  &&  \Bu  &&
\Bu  \ar@{-}[d]_{\ddm_1 z}   &&  \Bu   \ar@{-}[r]    \ar@{}[d]|{e_0}  &
\Bu   \ar@{-}[r] &   \Bu  &&
\\  
%line 4
\Bu  \ar@{-}[r]_{\ddp_2z}   &  \Bu  \ar@{-}[r]_{\ddp_2 u}   &
\Bu  \ar@{=}[rru]   &&&&
\Bu  \ar@{-}[r]_{\ddp_2 z}    \ar@{=}[rru] &
\Bu  \ar@{-}[r]_{\ddp_2 u}   &  \Bu \ar@{=}[rru] &&
}
\bigskip

	Moreover, the nullary interchange $ e_1(x) +_2 e_1(y) = e_1(x +_1 y) $ (of the {\it weak} cubical 
case (\ref{7.1.5})) is replaced with an {\it invertible} special $(n+1)$-map $ \io_1(x, y), $ which is defined 
when $ x, y $ are 1-consecutive, is natural on $n$-maps and has the following faces (partially 
displayed below)
    \begin{equation}
    \begin{array}{lr}
\io_1(x, y) \c  e_1(x) +_2 e_1(y)  \to  e_1(x +_1 y)   &   
 \mbox{({\it nullary interchange}),}
\\[3pt]
\ddm_1 \io_1(x, y)  =  \ddp_1 \io_1(x, y)  =  e_0(x +_1 y),  
\\[3pt]
\ddm_{j+1} \; \io_1(x, y)  =  e_0(e_1\ddm_j x),  &
\ddp_{j+1} \; \io_1(x, y)  =  e_0(e_1\ddp_j y),
    \label{7.3.6}\end{array}
    \end{equation}
     \xymatrix  @C=24pt @R=20pt  
{  
&&   \Bu  \ar@{-}[r]^{ e_1\ddm_jx }   \ar@{}[d]|{e_0}   &
\Bu     \ar@{-}[dd]^{x +_1 y}   &&&
\Bu  \ar@{-}[r]^{e_1\ddm_j x}   \ar@{-}[dd]   \ar@{}[rdd]|{e_1(x +_1 y)}  &
\Bu     \ar@{-}[dd]^{x +_1 y}  &&&
\\  
&  \Bu  \ar@{-}[r]   \ar@{-}[d]_x    \ar@{=}[ru]   \ar@{}[rd]|{e_1(x)}  &
\Bu  \ar@{=}[ru]  \ar@{-}[d]   \ar@{}[rd]|{e_0}  &&&  
\Bu \ar@{-}[d]_x    \ar@{=}[ru]  \ar@{}[rd]|{e_0}   &&&
\Bu  \ar[r]|>{~~~1}   \ar[d]^>{j+1}   \ar[ru]^>0  &
\\
& \Bu  \ar@{-}[r]   \ar@{-}[d]_y   \ar@{}[rd]|{e_1(y)}  & \Bu   \ar@{-}[d] &   
\Bu &&  \Bu  \ar@{-}[d]_y   & \Bu   \ar@{-}[r] \ar@{}[d]|{e_0}  &   \Bu  &&
\\  
& \Bu  \ar@{-}[r]_{e_1\ddp_j y}   & \Bu  \ar@{=}[ru]   &&&   
\Bu  \ar@{-}[r]_{e_1\ddp_j y}    \ar@{=}[ru]    & \Bu  \ar@{=}[ru] 
}
\smallskip

\Ndt (qcub.9)	Finally, these comparisons must satisfy some conditions of coherence (see \ref{7.4}).

\smallskip

	We say that $ \AA $ is unitary if the comparisons $ \la , \rho  $ are identities.

\subsection{Coherence}\label{7.4}
   Extending I.4.3, the coherence axiom (qcub.9) means that the following diagrams of 
transversal maps commute (assuming that all the cubical compositions make sense):
   
\bigskip
\ndt   (i) {\it coherence pentagon for} $ \ka  = \ka_1$  (writing $ + = +_1$):

%
% coherence pentagon
    \begin{equation}
    \begin{array}{c}  \xymatrix  @C=5pt @R=2pt  
{  
&\q&  ~x + (y + (z + u))~    \ar[lld]_{1 + \ka}   \ar[rdd]^\ka 
\\
~x + ((y + z) + u)~   \ar[dd]_\ka 
\\
&&&  ~(x + y)  +  (z + u)~   \ar[ldd]^\ka 
\\
~(x + (y + z)) + u~   \ar[rrd]_{\ka + 1} 
\\
&&   ~((x + y) + z) + u~
}
    \label{7.4.1} \end{array} \end{equation}

\bigskip
\ndt (ii) {\it coherence hexagon for $ \chi  = \chi_1$ and  $ \ka  = \ka_1$} (always writing $ + = +_1$):

%
% coherence hexagon
    \begin{equation}
    \begin{array}{c}  \xymatrix  @C=7pt @R=20pt  
{  
  (x + (y + z)) +_2 (x' + (y' + z'))~    \ar[r]^-{\ka \, +_2 \, \ka}  \ar[d]_{\chi}  
&   ~((x + y) + z) +_2 ((x' + y') + z')  \ar[d]^{\chi}
\\
(x +_2 x') + ((y + z) +_2 (y' + z'))~     \ar[d]_{1+\chi}  
&   ~((x + y) +_2 (x' + y')) + (z +_2 z')  \ar[d]^{\chi+1} 
\\
(x +_2 x') + ((y +_2 y') + (z +_2 z'))~     \ar[r]_\ka    
&   ~((x +_2 x') + (y +_2 y')) + (z +_2 z')
}
    \label{7.4.2} \end{array} \end{equation}

\bigskip
\ndt  (iii) {\it coherence conditions of the interchanges $ \chi  = \chi _1$,  $\io = \io_1 $ with the unit comparisons $ \la = \la_1 $ and $ \rho = \rho_1$}  (writing $ + = +_1$,  $ e = e_1$ and $\dda = \dda_1$): 

%
% two coherence conditions
    \begin{equation}
    \begin{array}{c}  \xymatrix  @C=8pt @R=25pt  
{  
  ~(e\ddm x +Êx) +_2 (e\ddm  y +Êy)~    \ar[rd]^-{\la \, +_2\,  \la}  \ar[d]_{\chi}  &&
~(x + e\ddp x) +_2 (y + e\ddp y)~    \ar[ld]_-{\rho \, +_2\,  \rho}  \ar[d]^{\chi}
\\
~(e\ddm x +_2Êe\ddm y) + (x  +_2Êy)~     \ar[d]_{\io + 1}  &
~x +_2 y~   &
~(x +_2 y) + (e\ddp x +_2  e\ddp y)~  \ar[d]^{1 + \io}
\\
~(e(\ddm x +Ê\ddm y) + (x  +_2Êy)~     \ar[ru]_-{\la}  
&&
~(x +_2 y) + e(\ddp x +_2  \ddp y)~  \ar[lu]^-\rho
}
    \label{7.4.3} \end{array} \end{equation}
\bigskip

\subsection{Comments}\label{7.5}
	The symmetric {\it quasi} cubical category $ \A $ is a symmetric {\it weak} cubical 
category if the unit-comparisons $ \la, \rho $ are invertible, the symmetric comparison $ \si_1 $ is an 
identity and a further coherence axiom holds:
   
\bigskip
\ndt (iv) {\it coherence triangle for} $ \la_1, \rho_1, \ka_1:$
%
% coherence triangle
    \begin{equation}
    \begin{array}{c}  \xymatrix  @C=10pt @R=15pt  
{  
~ x +_1 (e_1\ddm_1 y +_1 y)~    \ar[rr]^{\ka}  \ar[rd]_{1 +\la}  
&&   ~(x +_1 e_1\ddp_1 x) +_1 y~     \ar[ld]^{\rho + 1}
\\
&   ~ x +_1 y ~ 
}
    \label{7.5.1} \end{array} \end{equation}

	Notice that the latter does not hold in $ \CCOSP_*(\Top) $ nor in $ \CCOB_*(k), $ where the 
maps $ 1+\la $ and $ (\rho+1)\ka $ collapse different cylinders. In Part I, also the nullary interchange 
$ \io_1 $ was taken to be an identity - which works in the situations studied there; the present more 
general definition seems to be preferable, from a formal point of view.

\end{document}